\documentclass{amsart}

\usepackage[english]{babel}
\usepackage{enumerate}
\usepackage{amsmath,amsthm,amsfonts,amssymb,euscript,verbatim,bbm}
\usepackage{hyperref}
\usepackage{doi}

\usepackage{color}

\theoremstyle{plain}
\newtheorem{theorem}{Theorem}
\newtheorem{lemma}[theorem]{Lemma}
\newtheorem{definition}[theorem]{Definition}
\newtheorem{proposition}[theorem]{Proposition}
\newtheorem{remark}[theorem]{Remark}

\newcommand{\eqdef}{\overset{\mbox{\tiny{def}}}{=}}
\newcommand{\pv}{p}
\newcommand{\pZ}{\pv^0}
\newcommand{\qv}{q}
\newcommand{\qZ}{\qv^0}
\newcommand{\pM}{p^{\mu}}

\newcommand{\h}{\phi(|q|^2/2)}

\def\d{{\mathrm d}}
\def\e{{\varepsilon}}
\def\E{{\tilde{\varepsilon}}}
\newcommand{\pii}{\frac{p_i}{p^0}}
\newcommand{\pj}{\frac{p_j}{p^0}}

\newcommand{\qi}{\frac{q_i}{q^0}}
\newcommand{\qj}{\frac{q_j}{q^0}}
\newcommand{\qij}{\frac{q_i q_j}{(q^0)^2}}

\newcommand{\2}{\frac{q_2}{q^0}}
\newcommand{\3}{\frac{q_3}{q^0}}

\def\R {\mathbb R}
\def\Q {\mathcal Q}
\newcommand{\threed}{{{\mathbb R}^3}}
\newcommand{\rM}{\rho}
\newcommand{\qS}{\tau}

\newcommand{\TA}{a}
\newcommand{\TB}{b}

\DeclareMathOperator*{\esssup}{ess\,sup}

\begin{document}

\title[Entropy dissipation estimate for relativistic Landau]{Entropy dissipation estimates for the relativistic Landau equation, and applications}

\author[R. M. Strain]{Robert M. Strain}
\address{R. M. Strain, Department of Mathematics, University of Pennsylvania, Philadelphia, PA 19104, USA.}
\email{strain@math.upenn.edu}
\thanks{R.M.S. was partially supported by the NSF grants DMS-1500916 and DMS-1764177.}

\author[M. Taskovi\'{c}]{Maja Taskovi\'{c}}
\address{M. Taskovi\'{c}, Department of Mathematics, University of Pennsylvania, Philadelphia, PA 19104, USA.}
\email{taskovic@math.upenn.edu}


\begin{abstract}
In this paper we study the Cauchy problem for the spatially homogeneous relativistic Landau equation with Coulomb interactions. Despite it's physical importance, this equation has not received a lot of mathematical attention we think due to the extreme complexity of the relativistic structure of the kernel of the collision operator.  In this paper we first largely decompose the structure of the relativistic Landau collision operator.  After that we prove the global Entropy dissipation estimate.  Then we prove the propagation of any polynomial moment for a weak solution.  Lastly we prove the existence of a true weak solution for a large class of initial data.
\end{abstract}

\setcounter{tocdepth}{1}
\maketitle
\tableofcontents

\section{Introduction to the relativistic Landau equation}

\thispagestyle{empty}

Landau, in 1936, introduced a correction to the Boltzmann equation that is used to model a dilute hot plasma where fast moving particles interact via Coulomb interactions \cite{HintonArticle,MR684990}.  This widely used model, now called the Landau equation, does not include the effects of Einstein's theory of special relativity. When particle velocities are close to the speed of light, which happens frequently in a hot plasma, then relativistic effects become important. The relativistic version of Landau's equation was derived by Budker and Beliaev in 1956 \cite{MR0083886,BelyaevBudker}. It is a widely accepted fundamental model for describing the dynamics of a dilute collisional plasma.

The spatially homogeneous relativistic Landau equation is given by
\begin{equation} \label{rel.landau.eq}
\partial_t f =\mathcal{C}(f, f)
\end{equation}
with initial condition
$
f(0, \pv)=f_{0}( \pv).
$
A relativistic particle has momentum $\pv=(\pv^1,\pv^2,\pv^3)\in \threed$.
 The energy of a particle is given by
$
\pZ= \sqrt{1+|\pv|^2}
$
where $| \pv |^2 \eqdef \pv \cdot \pv $.  Let $g(p)$, $h(p)$ be two functions, then the relativistic Landau collision operator is defined by
\begin{equation}
\mathcal{C}(h,g)(p)
\eqdef
\partial_{p_i} \int_{\mathbb{R}^3}\Phi^{ij}(p,q)\left\{h(q) \partial_{p_j}  g(p)  -\partial_{q_j} h(q) g(p) \right\}
dq.
\label{landauC}
\end{equation}
Above and in the remainder of this article we will use the summation convention so that repeated indicies $i,j \in\{1,2,3\}$ are implicitly summed over without writing the sum $\sum_{i,j=1}^{3}$ notation.
The kernel is given by the $3\times 3$ non-negative matrix
\begin{eqnarray}
\Phi^{ij}(p,q)
&\eqdef&
\Lambda(p,q)S^{ij}(p,q),
\label{kernelnormalized}
\end{eqnarray}
The components of this kernel are defined  in \eqref{lambdaL} and \eqref{lambdaS}  below.  This kernel is the relativistic counterpart of the non-relativistic Landau kernel which is presented briefly in Remark \ref{landau.remark.classical}.  For notational simplicity and without loss of generality, in this paper, we will normalize all the physical constants to be one.

Solutions to the relativistic Landau equation formally satisfy the conservation of mass, total momentum and total energy  in integral form as
\begin{equation} \notag 
\int_{\mathbb{R}^3}   f(t,p) dp = \int_{\mathbb{R}^3} f_0(p) dp
\end{equation}
\begin{equation} \notag 
\int_{\mathbb{R}^3}  p f(t,p) dp = \int_{\mathbb{R}^3}  p f_0(p)dp
\end{equation}
\begin{equation} \notag 
\int_{\mathbb{R}^3}  \pZ f(t,p)dp = \int_{\mathbb{R}^3}  \pZ f_0(p) dp.
\end{equation}
Additionally the entropy of the relativistic Landau equation is defined as
\begin{equation}\label{entropyH}
H(t)=H(f(t))
\eqdef
\int_{ \mathbb{R}^3} f(t, p)\ln f(t, p)  dp.  
\end{equation}
Further the entropy dissipation is given by 
\begin{equation}\label{entropy.first.expression}
D(f) \eqdef 
-\int_{\mathbb{R}^3}
\mathcal{C}( f,f)(p)
\ln f(p) dp.
\end{equation}
Note that $D(f) \ge 0$ using the reformulation in \eqref{entropy.dissipation} together with \eqref{PositivePhi}.
Now using the entropy from \eqref{entropyH} and \eqref{entropy.first.expression} it can be calculated that solutions to \eqref{rel.landau.eq} formally satisfy
\begin{equation}\notag
\frac{d}{dt} H(f(t)) = - D(f(t)) \le 0.
\end{equation}
This is the Boltzmann H-Theorem for the relativistic Landau equation.  Further integrating  we have
\begin{equation}\label{sec.rel.landau.htheorem}
 H(f(T)) + \int_0^T D(f(t)) dt=  H(f_0) .
\end{equation}
This says that the entropy of solutions is non-increasing as time passes.   
Note that we define $H(f(t))$ with this sign to provide the above apriori estimate.

We also introduce the normalized relativistic Maxwellian as
\[
J(p)\eqdef \frac{1}{4\pi} e^{-\pZ}. 
\]
The relativistic Maxwellians, also known as the J\"{u}ttner solutions, are the equilibrium solutions to \eqref{rel.landau.eq}, and they are the extremizers of the entropy.

\subsection{Notation}
We will now define weighted $L^r$ spaces.
For all $l \in\mathbb{R}$, $r\in [1, +\infty]$, the weighted $L^r$ spaces and norms are defined as follows: 
\begin{equation}\label{lr.norm.def}
\|h\|_{L^r_l(\mathbb{R}^3)} \eqdef \| \langle \cdot \rangle^l  h\|_{L^r(\mathbb{R}^3)}
= \left(\int_{\mathbb{R}^3} \langle p \rangle^{l r}  \left| h(p)\right|^r \, dp \right)^{1/r}, 
\end{equation}
where
$
 \langle p \rangle \eqdef \left(1 + |p|^2 \right)^{1/2}
$
and $ L^r_l(\mathbb{R}^3) = \{h: \mathbb{R}^3 \to \mathbb{R}, \, \| h \|_{L^r_l(\mathbb{R}^3)} < +\infty \}$.  Further we let $L^r_0 = L^r$ when $l =0$.  We also use the standard definition for $L^{\infty}_l(\mathbb{R}^3)$.

Further for a non-negative function $f(t,p)\ge 0$  we define the energy
\begin{equation}\notag
E(f(t)) \eqdef \int_{\mathbb{R}^3} f(t,p) \pZ dp,
\end{equation}
and initially $f(0,p) = f_0(p)$ as
\begin{equation}\notag
E_0\eqdef
\int_{\mathbb{R}^3} f_0(p) \pZ dp.
\end{equation}
We also define the initial entropy as 
\begin{equation} \notag
H_0 \eqdef
 \int_{\mathbb{R}^3} f_0(p)  \log (f_0(p)) dp.
\end{equation}
And for a general function $h\ge0$ we define the absolute entropy functional
\begin{equation}\notag
\overline{H} (h) \eqdef   \int_{\mathbb{R}^3} h(p) \left|  \log h(p) \right| dp.
\end{equation}
We also define the following moment functional for $l\in\mathbb{R}$:
\begin{equation}\notag
M_l(h)   \eqdef   \int_{\mathbb{R}^3} \langle p \rangle^l h(p) dp.
\end{equation}
We further introduce the time dependent moment notation, for any $k\in \R$ and $T>0$, we measure the time dependent moments  as:
\begin{align}\label{momentTnotation}
M_k(f,T) \eqdef   \esssup_{t\in[0,T]} \int_\threed f(t,p) (1+|p|^2)^k dp.
\end{align}
We will also use the standard Sobolev spaces $\dot{H}^1(\threed)$ and $H^1(\threed)$ defined as:
$$
\| h\|_{\dot{H}^1(\threed)} = \| \nabla_p h\|_{L^2(\threed)},
\quad 
\| h\|_{H^1(\threed)}^2 = \| h\|_{L^2(\threed)}^2+ \| \nabla_p h\|_{L^2(\threed)}^2.
$$
We further use the notation $A \lesssim B$ to mean that there exists a positive inessential constant $C>0$ 
such that $A \le C B$.  When $A \lesssim B$ and $B \lesssim A$ then this is further denoted by $A \approx B$.

\subsection{Main Results}  
In this section we state the main results of the paper.  Our main result is Theorem \ref{entropy.thm} which states that the entropy dissipation \eqref{entropy.first.expression} controls uniformly the size of $\nabla \sqrt{f} \in L^2(\mathbb{R}^3)$.  Then in Theorem  \ref{weak.sol.thm} we prove the global existence of weak solutions for initial data $f_0 \in L^1_s \cap L \log L (\mathbb{R}^3)$ for any $s>1$.   Afterwards in Theorem \ref{prop-mom} we prove that weak solutions propagate high moments \eqref{momentTnotation} of any order.

We begin by stating the entropy dissipation estimate:

\begin{theorem}\label{entropy.thm}
Let  $f = f(p)\ge 0$ satisfy $M_1(f) \le \overline{M}$ and $\overline{H} (f) \le \overline{H}$ for some $\overline{M}>0$ and $\overline{H}>0$.

Then, there exists a positive constant which only depends (explicitly) on the mass $\int f\, dp$, the momentum $\int f\,p\,dp$, 
the energy $\int f\,\pZ \,dp$ and the upper bound on the entropy $\overline{H}$,  such that the following entropy dissipation inequality holds
$$ 
\int_{\mathbb{R}^3} |\nabla \sqrt{f(p)}|^2 \, dp  \lesssim  D(f)+1.
$$
\end{theorem}

This entropy dissipation estimate, which proves a gain of $\nabla \sqrt{f} \in L^2(\mathbb{R}^3)$, is the main theorem in our paper.  We will use this result to prove the global existence of a true weak solution, and also the propagation of high moment bounds.  Next we will give our definition of a weak solution to the relativistic Landau equation:

\begin{definition}\label{weak.sol.definition}
Fix any $T>0$.  
Let $f_0 \in L^1_1 \cap L \log L (\mathbb{R}^3)$ and $f =f(t,p)$ be a non-negative function satisfying 
$f\in L^\infty ([0,T]; L^1_1(\mathbb{R}^3_p))$ and
$\sqrt{f} \in L^2 ([0,T] ; H^1(\mathbb{R}^3_p))$.  Further suppose
$M_1(f(t,\cdot)) \le M_1(f_0)$ on $[0,T]$.  This function $f$ is called a weak solution of the relativistic Landau equation \eqref{rel.landau.eq}, \eqref{landauC},  and \eqref{kernelnormalized} on $[0,T]$ with initial data $f_0$ if for all $\varphi \eqdef \varphi(t,p) \in C_c^2([0,T] \times \mathbb{R}^3_p)$ it holds that
\begin{equation}\label{weak.formA.landau}
- \int_{\mathbb{R}^3} dp ~ f_0 \varphi(0,p) - 
\int_0^T dt ~\int_{\mathbb{R}^3} dp ~ f \partial_t\varphi
=
\int_0^T dt ~\int_{\mathbb{R}^3} dp ~ \mathcal{C}(f,f) \varphi.
\end{equation}
Here the integral on the right is defined by 
\begin{multline}\label{weak.formC.landau}
\int_{\mathbb{R}^3} dp ~
\mathcal{C}( f,f)
\varphi
\\
=
\frac{1}{2}\int_{\mathbb{R}^3}
 \int_{\mathbb{R}^3} f(q) f(p) \Phi^{ij}(p,q)
\left( \partial_{p_j}\partial_{p_i}\varphi(p) +  \partial_{q_j}\partial_{q_i}\varphi(q)\right) dqdp
\\
+\int_{\mathbb{R}^3}
 \int_{\mathbb{R}^3} f(p) f(q)
 \Lambda(p,q)  (\rM+2) \left(     q_i- p_i\right)
\left( \partial_{p_i}\varphi(p)  - \partial_{q_i}\varphi(q)\right) dqdp.
\end{multline}
\end{definition}

See the derivation of \eqref{weak.form.landau} in Section \ref{sec.weak.form} to obtain the weak form of the Landau collision operator \eqref{landauC} given in \eqref{weak.formA.landau} and \eqref{weak.formC.landau}.

Next we state the theorem which gives the existence of a weak solution to the relativistic Landau equation:

\begin{theorem}\label{weak.sol.thm}
Given initial data $f_0 \in L^1_s \cap L \log L (\mathbb{R}^3)$ for some $s>1$, there exists a weak solution to the Cauchy problem for the relativistic Landau equation.

Moreover for $\varphi \in W^{2,\infty}$ the mapping $t \to \int_{\threed} f(t) \varphi$ is H{\"o}lder continuous.
\end{theorem}

Lastly, we give the theorem which shows the propagation of any polynomial moment for a weak solution to the relativistic Landau equation:

\begin{theorem}\label{prop-mom}
Let $T>0$ and $k>1$. Suppose $f(t,p) \geq 0$ is a weak solution of the relativistic Landau equation on $[0,T]\times \threed$ associated to the initial data $f_0 \in L^1_1 \cap L\log L(\threed)$. Suppose also that the initial data satisfies $\int_\threed f_0(p) (1+|p|^2)^k dp < \infty$. Then the moment of order $2k$ of $f$ is bounded locally in time, that is, 
\begin{align}\notag
M_k(f,T) = \esssup_{t\in[0,T]}  \int_{\threed} f(t,p) (1+|p|^2)^k dp < C,
\end{align}
where $C>0$ is a finite constant depending only on $T$, the collision kernel $\Phi$, the inital mass, momentum, energy and entropy, $\Q_T(f)  \eqdef  \int_0^T \|f\|_{L^3(\threed)} dt$ and the initial moment  $\int_\threed f_0(p) (1+|p|^2)^k dp < \infty$.
\end{theorem}

Note that the finiteness of  $\Q_T(f)  \eqdef  \int_0^T \|f\|_{L^3(\threed)} dt\le C < \infty$ in Theorem \ref{prop-mom}  follows directly from the Entropy estimate in Theorem \ref{entropy.thm} and the Sobolev inequality.

Now in the next section we will give an overview of some previous results for the relativistic Landau equation, and the classical Landau equation.  We will later explain the structure of the classical Landau equation in Remark \ref{landau.remark.classical}.

\subsection{The literature}
We start by describing results for the relativistic Landau equation.  A detailed analysis of the linearized relativistic Landau collision operator was performed by Lemou in  \cite{MR1773932} in 2000.  In 2004 \cite{MR2100057}, Strain and Guo proved the global existence of unique classical solutions to the relativistic Landau-Maxwell system with initial data that is close to the relativistic Maxwellian equilibrium solution.  Then in 2006 \cite{MR2289548} Hsiao and Yu proved the existence of global classical solutions to the initial value problem for the simpler relativistic {L}andau equation with nearby relativistic Maxwellian initial data in the whole space.    In 2009 the $C^\infty$ smoothing effects were shown by Yu \cite{MR2514726} for the relativistic Landau-Maxwell system with nearby equilibrium initial data under the assumption that the electric and magnetic fields are infinitely smooth.  Further for relativistic Landau-Poisson equation the smoothing effects were shown in \cite{MR2514726} without additional assumptions.  In 2010 the Hypocoercivity of the relativistic Boltzmann and the relativistic Landau equations was proven in
\cite{MR2593052}, by Yang and Yu, including the optimal large time decay rates in $\mathbb{R}^3_x$.  In 2012, Yang and Yu, in \cite{MR2921603} the global in time classical solutions to the relativistic Landau-Maxwell system in the whole space $\mathbb{R}^3_x$ was proven for initial data which is nearby to the relativistic Maxwellian.   In 2014  
\cite{MR3121715} again looked at the Cauchy problem for the relativistic Landau-Maxwell system in $\mathbb{R}^3_x$.  In this paper for nearby Maxwellian initital data the optimal large time decay rates were proven.  Further see \cite{MR3369254}.  Then in 2015 Ha and Xiao in \cite{MR3391358} established the $L^2$ stability of the relativistic Landau equation and the non-relativistic Landau equation.   In 2016 the authors of \cite{MR3488578} studied the spectral structure of the linearized relativistic Landau equation in  $\mathbb{R}^3_x$ in the $L^2$ space.   In 2017 \cite{MR3622101} the authors did a precise spectral analysis of the relativistic Vlasov-Poisson-Landau equation in the whole space $\mathbb{R}^3_x$ and they used that to prove the optimal large time decay rates, including lower bounds on the decay rates.

The non-relativistic Landau equation has experienced a much larger amount of mathematical study in comparison.      We will mention only a small sample of results that are closely related to this paper.  Arsen'ev and Peskov in 1977 in \cite{MR0470442} proved the existence of a local in time bounded solution.  Then the uniqueness of bounded solutions with the Coulomb potential is shown in Fournier \cite{MR2718931} in 2010.  The uniqueness for soft potentials was previously shown in \cite{MR2502525} in 2009.  In 2002 \cite{MR1946444} Guo proved the global existence of classical solutions to the spatially dependent Landau equation with nearby Maxwellian equilibrium initial data.  The large time decay rates were shown in \cite{MR2209761}.    See also the recent developments in \cite{MR3625186} which study the case with a mild velocity tail on the initial data.   Further \cite{MR3670754} performs a numerical study on the large time decay rate in terms of the $2/3$ law as in \cite{MR2366140}.   See also \cite{MR2904573,MR3101794}.

Now in the spatially homogeneous situation, in \cite{MR1737547,MR1737548} Desvillettes and Villani proved the large data global well-posedness and smoothness of solutions for the Landau equation with hard potentials.  In \cite{MR1650006} Villani proved the existence of weak H-solutions of the spatially homogeneous Landau equation with Coulomb potential in 1998.  Then 2015 in \cite{MR3369941} Desvillettes proved an Entropy dissipation estimate for the Landau equation, and used it to conclude that the H-solutions are actually true weak solutions.  We use several of the methods from \cite{MR3369941} in the proofs in this paper, as described in detail below. Further developments can be found in \cite{MR3557719,MR3614751}.  Also \cite{MR3158719} proved $L^p$ estimates for the Landau equation with soft potentials.  In \cite{MR3375485} apriori estimates for the Landau equation with soft potentials including the Coulomb case are proven.    Recently also \cite{MR3582250} proves upper bounds for certain parabolic equations,  including the spatially dependent {L}andau equation by assuming that the local conservation laws are bounded.  And \cite{HarnackLandau} proves a Harnack inequality for solutions to kinetic Fokker-Planck equations with rough coefficients and applies that to the spatially dependent Landau equation to obtain a $C^\alpha$ estimate, assuming that the local conservation laws are bounded.  In Gualdani-Guillen \cite{MR3599518} estimates are proven for the homogeneous {L}andau equation with {C}oulomb potential.

In the quantum situation, Bagland \cite{MR2068110} in 2004  proved large data global well-posedness for the {L}andau-{F}ermi-{D}irac equation for hard potentials.  Also a related model problem for the Landau equation was introduced in  \cite{MR2901061}, which has been further studied in  \cite{MR2914961,MR3599518}.

In the next section we will give an overview of the methods used in our proofs.

\subsection{Overview of the proofs} The major new difficulties in the proofs of the Theorems \ref{entropy.thm} through \ref{prop-mom} are largely algebraic.   In particular the structure of the relativistic Landau kernel \eqref{kernelnormalized} with \eqref{lambdaL} and \eqref{lambdaS} causes several extreme mathematical algebraic difficulties.  This is initially seen in the proof of Lemma \ref{sij.ident.lem} below, where the non-negativity of the kernel \eqref{kernelnormalized} is given in two proofs.  This result is known \cite{MR1773932,MR684990}.   However our proofs are new, and they shed new light on the structure of the relativistic Landau kernel that allows us to perform the analysis in later sections.

We start by defining the following quantities $\rM$ and $\qS$ by
\begin{gather}
\rM
=
p^0 q^0 - p\cdot q-1 \ge 0.
\label{gDEFINITION}
\end{gather}
\begin{eqnarray}
\qS
=
 p^0 q^0 - p\cdot q+1.
\label{sDEFINITION}
\end{eqnarray}
Then the kernel takes the standard form \eqref{kernelnormalized}, 
$
\Phi^{ij}(p,q) =
\Lambda  S^{ij},
$
with
\begin{eqnarray}
\Lambda
&\eqdef&
\frac{( \rM+1)^2}{\pZ \qZ}
\left(  \qS \rM \right)^{-3/2},
 \label{lambdaL}
 \\
 \label{lambdaS}
S^{ij}
&\eqdef&
\qS \rM ~
\delta_{ij}
-
\left(p_{i}-q_{i}\right) \left(p_{j}-q_{j}\right)
+
\rM
\left(p_{i} q_{j}+p_{j} q_{i}\right).
\end{eqnarray}
Here in particular $\qS=\rM+2$.

Now a crucial point in our analysis is to introduce a new decomposition of $S^{ij}$ in \eqref{lambdaS} as a difference of two projections as 
$$
S^{ij} = P^{ij} - A^{ij}
$$
where
$$
S^{ij} = 
\qS \rM
\delta_{ij}
-
\left(p_{i}-q_{i}\right) \left(p_{j}-q_{j}\right)
+
\left( p^0 q^0 - p\cdot q-1 \right)
\left(p_{i} q_{j}+p_{j} q_{i}\right),
$$
with $\qS \rM = \left( (p^0 q^0 - p\cdot q)^2 - 1 \right)$.
Here
\begin{equation}\notag
P^{ij} \eqdef
\left| q^0p -  p^0 q \right|^2 \delta_{ij}
-
\left(q^0 p_{i}-p^0 q_{i}\right) \left(q^0 p_{j}-p^0 q_{j}\right),
\end{equation}
and
\begin{equation}\notag
A^{ij} \eqdef
\left| p \times q \right|^2 \delta_{ij}
-
|q|^2 p_{i}p_{j} 
-
|p|^2 q_{i}q_{j} 
+
(p \cdot q) 
\left( p_{i}q_{j}+ p_{j}q_{i}\right).
\end{equation}
This is shown in \eqref{pij} and \eqref{aij}.  This new complicated decomposition is the heart of our first proof of non-negativity of the kernel.    

This decomposition is then very helpful in our second proof of non-negativity of the kernel, because it enables us to write down the eigenvectors and eigenvalues of the relativistic Landau kernel \eqref{kernelnormalized} in \eqref{v1} - \eqref{v3} as far as we know for the first time.   (Note that the eigenvalues of the linearized relativistic Landau operator were given in \cite{MR1773932}, however these are very different and they are not for the kernel \eqref{kernelnormalized}.)  This eigenvalue decomposition of the kernel \eqref{kernelnormalized} directly gives us the second proof of positivity.  

Both of these decompositions described above are crucial to our poof of the entropy dissipation estimate from Theorem \ref{entropy.thm}.  The proof of Theorem \ref{entropy.thm} otherwise largely uses the method from Desvillettes in \cite{MR3557719,MR3369941}. In particular, we use the knowledge of eigenvalues and eigenvectors of the relativistic Landau kernel that comes from our decompositions to find the lower bound on $S^{ij} \xi_i \xi_j$ and more generally on the kernel $\Phi^{ij} \xi_i \xi_j$.
This leads to an auxiliary lower bound on the entropy dissipation $D(f)$ that will be crucially used later in the proof.  In order to obtain this auxiliary lower bound we use the representation of the entropy dissipation presented in \eqref{entropy.dissipation} in  Section \ref{sec.entropy.diss}, namely the entropy dissipation can be expressed as an integral of $\Phi^{ij} \xi_i \xi_j$  against $f(p) f(q)$ with the particular choice of 
 $\xi = \frac{\partial_p f(p)}{f} -  \frac{\partial_q f(q)}{f}$. Part of this lower bound contains a vector product 
$|(\qZ p - \pZ q) \times \xi|^2$, which is rewritten as $\sum |q_{ij}|^2$ with an appropriate choice of $q_{ij}$ thanks to the general  identity $|x \times y|^2 = \frac{1}{2} \sum_{i,j=1}^3 \left(x_i y_j - x_j y_i \right)^2$.
The proof proceeds by calculating three expressions
\begin{align*}
\notag
& \int q_{ij}(p,q) \,  \phi(\frac{|q|^2}{2}) \, f(q) \, dq, \\
& \int q_{ij}(p,q)  \, \qi \,  \phi(\frac{|q|^2}{2}) \,  f(q) \, dq,\\
& \int q_{ij}(p,q) \, \qj \, \phi(\frac{|q|^2}{2}) \,  f(q) \, dq,
\end{align*}
where $ \phi(\frac{|q|^2}{2})$ is a given test function. These three expressions can be thought of as a $3\times3$ system of equations with the unknowns 
$\pii \frac{\partial_{p_j}f}{f}(p) -  \pj \frac{\partial_{p_i}f}{f}(p) $, 
$\frac{\partial_{p_i}f}{f}(p)$ and 
$\frac{\partial_{p_j}f}{f}(p)$.
Now we use Cramer's rule to express $\frac{\partial_{p_i}f}{f}(p)$ (one of the unknowns of the system). From  there, one uses elementary inequalities to obtain a pointwise upper bound of $\left|\frac{\partial_{p_i}f}{f}(p)\right|^2$ and consequently of the integral
 $\int f(p)\left|\frac{\partial_{p_i}f}{f}(p)\right|^2 dp$. From that point using the Cauchy-Schwartz inequality on one of the terms will lead to the expression of the auxiliary lower bound on the entropy dissipation that can then be bounded from above by the entropy dissipation.

 Lemma \ref{determinant} provides the key estimate needed to conclude Theorem \ref{entropy.thm}, namely a lower bound on a determinant $\Delta_\phi(f)$ given in the statement of the lemma. The inverse of this determinant (and thus the need for the lower bound) naturally comes into play due to the use of Cramer's rule in the proof of Theorem \ref{entropy.thm}. This determinant resembles the ones appearing in Desvillettes \cite{MR3557719,MR3369941}.   The difference is that the entries in our determinant are relativistic quantities $\frac{q_i}{\qZ}$ (as opposed to simply $q_i$). This results in a series of extremely complicated algebraic expressions. 
 
 Here we summarize the challenges and strategy. The idea is to diagonalize the determinant that is showing up inside the integral defining $\Delta_\phi(f)$ to get a lower bound on $\Delta_\phi(f)$ in terms of the following quantity
$$
 \sup_{\{\lambda^2 + \mu^2 +\nu^2 =1\}} \int_{B(0,R)} f(q) \,  \chi_{\left| \lambda + \mu \qi + \nu \qj \right| < \e} \, \d q.
$$
Using the fact that entropy is bounded by $\overline{H}$ and that the domain is bounded, for any constant $A$ this integral can be estimated by 
$$
 \frac{\overline{H}}{\ln A} \, + \, A \sup_{\{\lambda^2 + \mu^2 +\nu^2 =1\}} Y_{\{\lambda, \mu, \nu, R, \e\}}, 
$$
where 
\begin{align*}
Y_{\{\lambda, \mu, \nu, R, \e\}}  = \int_{B(0,R)} \chi_{\left| \lambda + \mu \2 + \nu \3 \right| < \e} \, \d q
 =   \int_{B(0,R)} \chi_{\left| \tilde{\lambda} + \tilde{\mu} \frac{q_2}{\qZ} \right| < \e} \, \d q,
\end{align*}
for $\lambda^2 + \mu^2 +\nu^2 =1$ and $\tilde{\lambda}^2 + \tilde{\mu}^2 =1$. The second equality can be obtained by rotating the coordinate system. Estimating 
$ \int_{B(0,R)} \chi_{\left| \tilde{\lambda} + \tilde{\mu} \frac{q_2}{\qZ} \right| < \e} \, \d q$ is quite complex because the variable $q$ appears via $\qZ$ in the denominator of  the expression defining the domain of the characteristic function. The way we overcome this difficulty is by chipping away the values of $|\mu|$ for which this integral is zero in a series of splitting regimes. Eventually, one shows that the set $\{q: \left| \lambda + \mu \frac{q_2}{\qZ} \right| \leq \e\} \cap B(0,R)$ is non-empty only when 
$$
|\mu| \geq \sqrt{\frac{1+\delta}{2}},
$$
where $\delta= \delta(R)$ is a fixed number depending on the radius $R$. This bound on $\mu$ will be important in that it guarantees that the  expressions $ \frac{-\e}{|\mu|} \pm \sqrt{\frac{1}{\mu^2}-1}$ and $\frac{\e}{|\mu|} \pm \sqrt{\frac{1}{\mu^2}-1}$ are both less than one. These expressions show up in the following representation of the set
$ \{q: \left| \lambda + \mu \frac{q_2}{\qZ} \right| \leq \e\}$
\begin{align*}
\left\{q \, : \,    \frac{-\e}{|\mu|} \pm \sqrt{\frac{1}{\mu^2}-1} <  \2<   \frac{\e}{|\mu|} \pm \sqrt{\frac{1}{\mu^2}-1}  \right\}, 
\end{align*}
which is then used to conclude  that this set lies between two rotating curves, which in turn is used to obtain an estimate on 
$Y_{\{\lambda, \mu, \nu, R, \e\}}$, and thus the determinant $\Delta_\phi(f)$.

We will now say few words on the proof of Theorem \ref{prop-mom}. The propagation of moments is proven by inductively invoking Lemma \ref{lem-mom}, which says that if the moment of order $2k-1$ is finite up to a time $T$ (i.e. $M_{k-\frac{1}{2}}(f,T) <\infty$) and if the moment of order $2k$ is finite initially, then the moment of order $2k$ stays finite up to the time $T$. To prove this lemma, one uses the weak formulation of the relativistic Landau equation with the test function which is obtained by a smooth cutoff of the polynomial weight. The right-hand side of the weak formulation with this particular test function is then broken into three subdomains depending on the size of $\pZ$ and $\qZ$. Depending on the case, one then uses Young and H\"older inequalities, where the parameters of the corresponding $L^p$ spaces are chosen so that terms can be estimated by $\int_0^T \|f(t,\cdot)\|_{L^3(\threed)}dt$, which is a finite quantity thanks to the entropy dissipation estimate from Theorem \ref{entropy.thm} and the Sobolev embedding.

Note that in Section \ref{sec:uniformaijBOUNDS} we prove uniform upper and lower bounds for the diffusion matrix $a^{ij}(h)$ in \eqref{aijhdef} assuming only that the conserved quantities are bounded.

Lastly in Section \ref{sec:TWS}, we prove the global existence of a weak solution to the relativistic Landau equation.  The construction is rather standard along the lines of \cite{MR1650006,MR3369941,MR2068110}.  In the next section we will outline the rest of this article.

\subsection{Outline of the remainder of this article}  The rest of this paper is organized as follows.  In Section \ref{sec.structure} we explain the detailed complex structure of the relativistic Landau collision operator \eqref{landauC} and its kernel \eqref{kernelnormalized}.   In particular we will derive the weak formulation of the relativistic Landau equation.  And we reformulate the entropy dissipation \eqref{entropy.first.expression} as in \eqref{entropy.dissipation}.  After that we give two direct proofs of the pointwise non-negativity of the kernel.  Further we explain how to express the collision operator in non-conservative form.  
Then in Section \ref{sec.entropy.est} we prove the entropy dissipation estimate from Theorem \ref{entropy.thm}.
Following that in Section \ref{sec.prop.high.moment} we prove the propagation of high moment bounds.  
Lastly in Section \ref{sec:TWS} we prove the global existence of a true weak solution to the relativistic Landau equation.

\section{Structure of the relativistic Landau equation}\label{sec.structure}

In this section we explain in depth the structure of the relativistic Landau collision operator \eqref{rel.landau.eq}.  In Section \ref{sec.landau.consrv} we explain the conservative form of the collision operator.  Then in Section \ref{sec.weak.form} we will derive the weak form of the relativistic Landau equation.   Then in Section \ref{sec.entropy.diss} we discuss the entropy dissipation estimate.    After that in Section \ref{non.neg.proof2}, we will give two direct proofs of the non-negativity of the kernel, as in \eqref{PositivePhi}.  Then finally in Section \ref{landau.non.conservative.form} we explain the non-conservative form of the relativistic Landau operator.

It is known that the collision kernel $\Phi$, from \eqref{kernelnormalized} with \eqref{lambdaL} and \eqref{lambdaS}, is a non-negative  matrix satisfying
\begin{equation}
\sum_{i=1}^3 \Phi^{ij}(p,q)\left( \frac{q_i}{\qZ}- \frac{p_i}{\pZ}\right)
=
\sum_{j=1}^3 \Phi^{ij}(p,q)\left( \frac{q_j}{\qZ}- \frac{p_j}{\pZ}\right)
=
0,
\label{nullphi}
\end{equation}
and \cite{MR1773932,MR684990}
\begin{equation}
\sum_{i,j}\Phi^{ij}(p,q) w_i w_j >0 ~~ \mbox{if}~~ w\ne d\left( \frac{p}{\pZ}-\frac{q}{\qZ}\right)
~~\forall d\in\mathbb{R}.
\label{PositivePhi}
\end{equation}
This property represents the physical assumption that grazing collisions dominate.  In particular the momentum of colliding particles is orthogonal to their relative velocity. This is also a key property used to derive the conservation laws and the entropy dissipation.

It follows from (\ref{nullphi}) that for any smooth decaying function $g(p)$ we have
$$
\int_{\mathbb{R}^3} dp ~ 
 \begin{pmatrix}
      1   \\      p  \\ \pZ
\end{pmatrix}\mathcal{C}(g, g)(p)
=
0.
$$
In particular, after integrating by parts and  using \eqref{nullphi}, we have
\begin{multline}\label{zero.calc.integration}
\int_{\mathbb{R}^3} dp ~
 \pZ
\mathcal{C}(g, g)(p)
\\
=
-
\frac{1}{2} \int_{\mathbb{R}^3} dp ~ \frac{p_i}{\pZ}
 \int_{\mathbb{R}^3}\Phi^{ij}(p,q)\left\{g(q) \partial_{p_j}  g(p)  -\partial_{q_j} g(q) g(p) \right\}
dq
\\
-
\frac{1}{2} \int_{\mathbb{R}^3} dq ~ \frac{q_i}{\qZ}
 \int_{\mathbb{R}^3}\Phi^{ij}(p,q)\left\{g(p) \partial_{q_j}  g(q)  -\partial_{p_j} g(p) g(q) \right\}
dp
\\
=
-
\frac{1}{2} \int_{\mathbb{R}^3} dp ~ 
 \int_{\mathbb{R}^3}\left(\frac{p_i}{\pZ} - \frac{q_i}{\qZ}  \right)\Phi^{ij}(p,q)\left\{g(q) \partial_{p_j}  g(p)  -\partial_{q_j} g(q) g(p) \right\}
dq
\\
=0.
\end{multline}
The other cases follow similarly.  Then these identities lead directly to the conservation laws above 
\eqref{entropyH}.

\subsection{Landau operator in conservative form}\label{sec.landau.consrv}
In this section we will express the Landau operator in conservative form.  First we recall a lemma from \cite{MR2100057}:

\begin{lemma}\label{2derivatives}   
We compute a sum of first derivatives in $q$ of (\ref{kernelnormalized}) as
\begin{equation}
 \partial_{q_j}\Phi^{ij}(p,q)
=
2 \Lambda(p,q) \left(  (\rM+1)   p_i- q_i\right).
\label{intermediate2}
\end{equation}
This term has a second order singularity at $p=q$.   We further compute a sum of (\ref{intermediate2}) over first derivatives in $p$ as 
\begin{equation}
\partial_{p_i}\partial_{q_j}\Phi^{ij}(p,q)
=
4\frac{ (\rM+1) }{ \pZ  \qZ }\left( \qS\rM \right)^{-1/2}\ge 0, \quad p\ne q.
\notag \label{surprise}
\end{equation}
This term has a first order singularity. 
\end{lemma}

Note that there is actually a dirac mass hiding in $\partial_{p_i}\partial_{q_j}\Phi^{ij}(p,q)$ when $p=q$ as can be seen in Lemma \ref{posIT}, which is proven in \cite{MR2100057}.

\begin{remark}\label{landau.remark.classical}
We note that the above is very different from the non-relativistic theory.  The following non-relativistic Landau collision operator (with normalized constants) is given by
$$
\mathcal{C}_{cl}(G, F)
\eqdef
  \nabla _v\cdot \left\{\int_{{\mathbb R}^3}\phi
(v-v^{\prime })\left\{\nabla _v G(v) F(v^{\prime })-G(v)\nabla _{v^\prime}F(v^{\prime
})\right\}dv^{\prime }\right\}.
$$
The non-negative $3\times 3$ matrix is
\begin{equation}
\phi ^{ij}(v)=\left\{ \delta _{ij}-\frac{v_iv_j}{|v|^2}\right\} \frac 1{ |v|}.  
\notag \label{classicalKERNEL}
\end{equation}
Then the derivatives of the classical kernel are as follows
$$
\partial_{v_i}\partial_{v^\prime_j}\phi^{ij}(v-v^\prime)=0, \quad  v\ne v'.
$$
This also contains a delta function when $v = v'$.
\end{remark}

We now define the notation
\begin{equation}\label{phi.not}
\Phi^{ij}(h)
=
\Phi^{ij}(h)(p)
\eqdef
\int_{\mathbb{R}^3}\Phi^{ij}(p,q)h(q)dq.
\end{equation}
We further use this notation as
$$
(\partial_{p_i} \Phi^{ij})(h)(p)
\eqdef
\int_{\mathbb{R}^3}\partial_{p_i} \Phi^{ij}(p,q)h(q)dq.
$$

Now directly from the collision operator from \eqref{landauC} with kernel \eqref{kernelnormalized}, \eqref{lambdaL} and \eqref{lambdaS} we can read off the expression of the Landau operator in conservative form
\begin{equation}\label{landauConservativeForm} 
\mathcal{C}( h,g)(p)
=
\partial_{p_i}\left( a^{ij}(h) \partial_{p_j}  g(p)
+
(\partial_{q_j} \Phi^{ij})(h)
g(p)
\right).
\end{equation}
Here we use the notation  \eqref{aijhdef}, and we recall \eqref{phi.not} and \eqref{intermediate2}.  


\subsection{Weak formulation of the relativistic Landau equation}\label{sec.weak.form}
We will now derive the weak formulation of the relativistic Landau collision operator \eqref{landauC}. 
For a test function $\phi(p)$, after integration by parts, using $(p,q)$ symmetry we have 
\begin{gather}\notag
\int_{\mathbb{R}^3}
\mathcal{C}( h,g)(p)
\phi(p) dp
=
-\int_{\mathbb{R}^3}
 \int_{\mathbb{R}^3}\Phi^{ij}(p,q)\left\{h(q) \partial_{p_j}  g(p)  -\partial_{q_j} h(q) g(p) \right\}
\partial_{p_i}\phi(p)  dqdp
\\
\label{symmetryC}
=
-
\frac{1}{2}\int_{\mathbb{R}^3}
 \int_{\mathbb{R}^3}\Phi^{ij}(p,q)\left\{h(q) \partial_{p_j}  g(p)  -\partial_{q_j} h(q) g(p) \right\}
\partial_{p_i}\phi(p)  dqdp
\\
\notag
-
\frac{1}{2}\int_{\mathbb{R}^3}
 \int_{\mathbb{R}^3}\Phi^{ij}(p,q)\left\{h(p) \partial_{q_j}  g(q)  -\partial_{p_j} h(p) g(q) \right\}
\partial_{q_i}\phi(q)  dqdp.
\end{gather}
Then after further integration by parts
\begin{gather*}
\int_{\mathbb{R}^3}
\mathcal{C}( h,g)(p)
\phi(p) dp
=
\frac{1}{2}\int_{\mathbb{R}^3}
 \int_{\mathbb{R}^3}\Phi^{ij}(p,q)h(q)   g(p)
\left( \partial_{p_j}\partial_{p_i}\phi(p) +  \partial_{q_j}\partial_{q_i}\phi(q)\right) dqdp
\\
-
\frac{1}{2}\int_{\mathbb{R}^3}
 \int_{\mathbb{R}^3} h(p) g(q)
\left( \partial_{q_j}\Phi^{ij}(p,q) \partial_{p_i}\phi(p) +   \partial_{p_j}\Phi^{ij}(p,q) \partial_{q_i}\phi(p)\right) dqdp
\\
+
\frac{1}{2}\int_{\mathbb{R}^3}
 \int_{\mathbb{R}^3} h(p) g(q)
\left( \partial_{p_j}\Phi^{ij}(p,q) \partial_{p_i}\phi(p) +   \partial_{q_j}\Phi^{ij}(p,q) \partial_{q_i}\phi(p)\right) dqdp.
\end{gather*}
This is a weak formulation of the Landau operator, but it can be further simplified.

By collecting terms, we will use the following weak formulation:
\begin{multline}\label{weak.form.landau.kernel}
\int_{\mathbb{R}^3}
\mathcal{C}( h,g)(p)
\phi(p) dp
\\
=
\frac{1}{2}\int_{\mathbb{R}^3}
 \int_{\mathbb{R}^3} h(q) g(p) \Phi^{ij}(p,q)
\left( \partial_{p_j}\partial_{p_i}\phi(p) +  \partial_{q_j}\partial_{q_i}\phi(q)\right) dqdp
\\
+\int_{\mathbb{R}^3}
 \int_{\mathbb{R}^3} h(p) g(q)
 \left(\partial_{p_j}\Phi^{ij}(p,q) - \partial_{q_j}\Phi^{ij}(p,q)  \right)
\left( \partial_{p_i}\phi(p)  - \partial_{q_i}\phi(q)\right) dqdp.
\end{multline}
This will be useful for studying the weak formulation of approximate problem later on \eqref{truncatedCauchy.n}.  Note that there is additional cancellation in $\partial_{p_j}\Phi^{ij}(p,q) - \partial_{q_j}\Phi^{ij}(p,q)$.

More precisely, in the specific case of \eqref{kernelnormalized}, from \eqref{intermediate2}, we have the simplification
\begin{equation}
 \partial_{q_j}\Phi^{ij}(p,q)
=
2 \Lambda(p,q) \left(  (\rM+1)   p_i- q_i\right),
\quad 
 \partial_{p_j}\Phi^{ij}(p,q)
=
2 \Lambda(p,q) \left(  (\rM+1)   q_i- p_i\right).
\notag
\end{equation}
We plug this in  to obtain that the simplified weak form of the relativistic Landau operator is
\begin{multline}\label{weak.form.landau}
\int_{\mathbb{R}^3}
\mathcal{C}( h,g)(p)
\phi(p) dp
\\
=
\frac{1}{2}\int_{\mathbb{R}^3}
 \int_{\mathbb{R}^3} h(q) g(p) \Phi^{ij}(p,q)
\left( \partial_{p_j}\partial_{p_i}\phi(p) +  \partial_{q_j}\partial_{q_i}\phi(q)\right) dqdp
\\
+\int_{\mathbb{R}^3}
 \int_{\mathbb{R}^3} h(p) g(q)
 \Lambda(p,q)  (\rM+2) \left(     q_i- p_i\right)
\left( \partial_{p_i}\phi(p)  - \partial_{q_i}\phi(q)\right) dqdp.
\end{multline}
Notice that both integrals have a first order singularity in the integrand when $p=q$.

\subsection{Entropy dissipation for the relativistic Landau equation}\label{sec.entropy.diss}
In this section we derive several representations for the entropy dissipation of the relativistic Landau equation.  We recall that the entropy dissipation is given by \eqref{entropy.first.expression}.

We plug \eqref{entropy.first.expression} into \eqref{symmetryC} to formally obtain
\begin{gather*}
D(f) 
=
\frac{1}{2}\int_{\mathbb{R}^3}
 \int_{\mathbb{R}^3}\Phi^{ij}(p,q)\left\{f(q) \partial_{p_j}  f(p)  -\partial_{q_j} f(q) f(p) \right\}
\frac{\partial_{p_i}f(p)}{f(p)}  dqdp
\\
+
\frac{1}{2}\int_{\mathbb{R}^3}
 \int_{\mathbb{R}^3}\Phi^{ij}(p,q)\left\{f(p) \partial_{q_j}  f(q)  -\partial_{p_j} f(p) f(q) \right\}
\frac{\partial_{q_i}f(q)}{f(q)}  dqdp.
\end{gather*}
We conclude the following formula for the entropy dissipation
\begin{gather}\notag
\frac{1}{2}\int_{\mathbb{R}^3}
 \int_{\mathbb{R}^3} f(p) f(q) \Phi^{ij}(p,q)
 \left( \frac{\partial_{p_j}f}{f}(p)  - \frac{\partial_{q_j}f}{f}(q)  \right)
 \left( \frac{\partial_{p_i}f}{f}(p)  - \frac{\partial_{q_i}f}{f}(q)  \right) dqdp
\\
\notag
=
2\int_{\mathbb{R}^3}
 \int_{\mathbb{R}^3}  \Phi^{ij}(p,q)
\left( \left( \partial_{p_j}  - \partial_{q_j}  \right) \sqrt{f(p) f(q)} \right)
\left(   \left( \partial_{p_i}  - \partial_{q_i}  \right) \sqrt{f(p) f(q)} \right)
dqdp
\\
=D(f)  \ge 0.
\notag
\end{gather}
Indeed we can take the following as the definition of $D(f)$:
\begin{gather}\notag
2\int_{\mathbb{R}^3}
 \int_{\mathbb{R}^3}  \Phi^{ij}(p,q)
\left( \left( \partial_{p_j}  - \partial_{q_j}  \right) \sqrt{f(p) f(q)} \right)
\left(   \left( \partial_{p_i}  - \partial_{q_i}  \right) \sqrt{f(p) f(q)} \right)
dqdp
\\
\eqdef D(f).
\label{entropy.dissipation}
\end{gather}
This expression \eqref{entropy.dissipation} will be used in the construction of weak solutions.

\subsection{Direct proof of pointwise non-negativity of the Kernel} \label{non.neg.proof2}

In this subsection we would like to give an alternative direct proof of \eqref{PositivePhi}.  Note that there is no proof given in \cite{MR684990} although the result is stated.  And the proof of \eqref{PositivePhi} in \cite{MR1773932} uses a complicated change of variable.  Here we give two direct proofs that can be expressed in the original coordinate system.  In particular we will see that the details of both proofs are useful in the later sections of the paper.

\begin{lemma}\label{sij.ident.lem}  For $\Phi^{ij}$ defined in \eqref{kernelnormalized}, we have 
$
\Phi^{ij} \xi_i \xi_j   \ge 0.
$
\end{lemma}

We will give two different direct proofs of this lemma.  The reason  is because they give two different useful expressions for 
$\Phi^{ij} \xi_i \xi_j$.

To begin a discussion of the first proof, we notice first that we can decompose $S^{ij}$ from \eqref{kernelnormalized} and \eqref{lambdaS}  as follows
$$
S^{ij} = P^{ij} - A^{ij}
$$
where recall from \eqref{lambdaS} that we have 
$$
S^{ij} = 
\qS \rM
\delta_{ij}
-
\left(p_{i}-q_{i}\right) \left(p_{j}-q_{j}\right)
+
\left( p^0 q^0 - p\cdot q-1 \right)
\left(p_{i} q_{j}+p_{j} q_{i}\right),
$$
where $\qS \rM = \left( (p^0 q^0 - p\cdot q)^2 - 1 \right)$.
Then further 
\begin{equation}\label{pij}
P^{ij} \eqdef
\left| q^0p -  p^0 q \right|^2 \delta_{ij}
-
\left(q^0 p_{i}-p^0 q_{i}\right) \left(q^0 p_{j}-p^0 q_{j}\right),
\end{equation}
and
\begin{equation}\label{aij}
A^{ij} \eqdef
\left| p \times q \right|^2 \delta_{ij}
-
|q|^2 p_{i}p_{j} 
-
|p|^2 q_{i}q_{j} 
+
(p \cdot q) 
\left( p_{i}q_{j}+ p_{j}q_{i}\right).
\end{equation}
This can be seen by direct pointwise comparison.  In particular  we observe that 
\begin{equation}\label{taurho.exp}
\qS \rM =
(p^0 q^0 - p\cdot q)^2 - 1 = \left| q^0p -  p^0 q \right|^2 - \left| p \times q \right|^2.
\end{equation}
We will  study $P^{ij} - A^{ij}$, this will be a crucially important expression in several places during the rest of this paper.   In particular, recalling \eqref{lambdaS}, then the first proof below will provide us with the formula
\begin{equation}\label{sij.ident}
S^{ij} \xi_i \xi_j = | (\qZ p - \pZ q)\times \xi|^2 - |(p\times \xi) \times (q \times \xi) |^2 \ge 0. 
\end{equation}

\begin{proof}[First proof of Lemma \ref{sij.ident.lem}]
Because of the structure of \eqref{kernelnormalized}, it will be sufficient to prove the pointwise identity \eqref{sij.ident} and the positivity of $S^{ij} \xi_i \xi_j$.  

To establish the identity \eqref{sij.ident}, first of all clearly
$$
P^{ij}\xi_i \xi_j  =
\left| q^0p -  p^0 q \right|^2 |\xi|^2
-
\left( \left(q^0 p-p^0 q\right)\cdot \xi\right)^2= | (\qZ p - \pZ q)\times \xi|^2.
$$
Then we will also show that 
$
A^{ij} \xi_i \xi_j  =  |(p\times \xi) \times (q \times \xi) |^2,
$
however this is more involved.  Note that directly
\begin{gather}\notag
A^{ij} \xi_i \xi_j  = 
\left| p \times q \right|^2 |\xi|^2
-
|q|^2 (p\cdot \xi)^2
-
|p|^2 (q\cdot \xi)^2
+
2(p \cdot q) 
(p\cdot \xi)(q\cdot \xi)
\\ \notag
= |p|^2 |q|^2 |\xi|^2 \left( \sin^2 \theta_1 - \cos^2 \theta_2- \cos^2 \theta_3+2\cos \theta_1\cos \theta_2 \cos \theta_3 \right). 
\end{gather}
Here for $\theta_i \in [0,\pi]$ and $(i=1,2,3)$, we used the definitions
$$
p\cdot q = |p||q| \cos\theta_1,
\quad 
p\cdot \xi = |p| |\xi| \cos\theta_2,
\quad 
q\cdot \xi = |q| |\xi|  \cos\theta_3.
$$
Then further define the angle $\psi \in [0,\pi]$ by
$$
|p\times \xi | |q \times \xi | \cos \psi \eqdef   (p\times \xi) \cdot (q \times \xi). 
$$
Then by the vector identity $(A \times B)\cdot (C \times D) = (A \cdot C) (B\cdot D) - (B \cdot C) (A \cdot D)$ with $A=p$, $C=q$, $B=D=\xi$ we can deduce the angle identity
$$
\sin\theta_2\sin\theta_3 \cos\psi = \cos\theta_1- \cos\theta_2\cos\theta_3.
$$
Now we calculate using only trig identities that 
\begin{multline}\notag
\mathcal{A} \eqdef \sin^2 \theta_1 - \cos^2 \theta_2- \cos^2 \theta_3+2\cos \theta_1\cos \theta_2 \cos \theta_3
\\ 
=1- \cos^2 \theta_1 - \cos^2 \theta_2- \cos^2 \theta_3+2\cos \theta_1\cos \theta_2 \cos \theta_3
\\ 
=1- \left(\sin\theta_2\sin\theta_3 \cos\psi +\cos\theta_2\cos\theta_3 \right)^2 - \cos^2 \theta_2- \cos^2 \theta_3
\\ 
+2\left(\sin\theta_2\sin\theta_3 \cos\psi +\cos\theta_2\cos\theta_3 \right)\cos \theta_2 \cos \theta_3
\\ 
=1- \sin^2\theta_2\sin^2\theta_3 \cos^2\psi +\cos^2\theta_2\cos^2\theta_3  - \cos^2 \theta_2- \cos^2 \theta_3
\\ 
=\sin^2\theta_2\sin^2\theta_3 - \sin^2\theta_2\sin^2\theta_3 \cos^2\psi 
\\ 
=\sin^2\theta_2\sin^2\theta_3 \sin^2\psi.
\end{multline}
We use this calculation to obtain the desired expression
\begin{gather}\notag
A^{ij} \xi_i \xi_j  
= |p|^2 |q|^2 |\xi|^2 \mathcal{A} 
= |p|^2 |q|^2 |\xi|^2 \sin^2\theta_2\sin^2\theta_3 \sin^2\psi
=  |(p\times \xi) \times (q \times \xi) |^2.
\end{gather}
One could also establish this equality using vector identities.

Now that we have the identity \eqref{sij.ident}, we will finish the proof by showing that the expression is positive.  We expand it out and use the angles defined previously
\begin{multline}\notag
S^{ij} \xi_i \xi_j = | (\qZ p - \pZ q)\times \xi|^2 - |(p\times \xi) \times (q \times \xi) |^2
\\ \notag
= (\qZ)^2  |p \times \xi|^2 + (\pZ)^2  |q \times \xi|^2 - 2 \pZ \qZ  |p \times \xi| |q \times \xi| \cos\psi
- |p \times \xi|^2 |q \times \xi|^2 \sin^2\psi
\\ \notag
\ge 
 (\qZ)^2  |p \times \xi|^2 \cos^2 \psi + (\pZ)^2  |q \times \xi|^2 - 2 \pZ \qZ  |p \times \xi| |q \times \xi| \cos\psi
\\ \notag
= \left(  (\qZ)  |p \times \xi| \cos \psi  -(\pZ)  |q \times \xi| \right)^2 \ge 0.
\end{multline}
Above we used the inequality $|p \times \xi|^2 |q \times \xi|^2 \sin^2\psi \le (\qZ)^2 |p \times \xi|^2 \sin^2\psi$.  
\end{proof}

For the second proof of Lemma \ref{sij.ident.lem} we will look at the eigenvalues.  Since by \eqref{nullphi} the null space of $S^{ij}$ is a span of vector $\left( \frac{p}{\pZ} - \frac{q}{\qZ}\right)$, the first eigenvalue of $S^{ij}$ is zero with eigenvector $ \left( \frac{p}{\pZ} - \frac{q}{\qZ}\right)$.   The matrix $S^{ij}$ is real-valued and symmetric, so its eigenvectors are orthogonal. One can then see that $p\times q$ (which is orthogonal to both $p$ and $q$, and thus to $v_1$) is another eigenvector (when $p$ and $q$ are not co-linear). Its eigenvalue can be calculated to be $\lambda_2 =|\qZ p - \pZ q|^2 - |p\times q|^2$.

To find the third eigenvector we will need to use some thoughtful guesswork.   One can look for it in the form of a linear combination $v_3=Ap + Bq$. As such it will automatically be orthogonal to $v_2$, so one just needs to impose the requirement that it is orthogonal to $v_1$. After some calculation, this leads to the third eigenvector being expressed as $\left( \qZ (p\cdot q) - \pZ |q|^2 \right)p   +   \left(\pZ (p\cdot q) - \qZ|p|^2 \right) q $ and the corresponding eigenvalue $\lambda_3  = |\qZ p - \pZ q|^2$.  In summary, the set of (normalized) eigenvectors and eigenvalues of $S^{ij}$ is the following:
\begin{align} 
v_1 &=\frac{\frac{p}{\pZ} - \frac{q}{\qZ}}{|\frac{p}{\pZ} - \frac{q}{\qZ}|}, 
\qquad \lambda_1  = 0, \label{v1}  \\
& \nonumber \\
v_2 &= \frac{p \times q}{|p \times q|}, \qquad
\lambda_2 = |\qZ p - \pZ q|^2 - |p\times q|^2,   \label{v2}\\
& \nonumber \\
v_3 &=\frac{\left( \qZ (p\cdot q) - \pZ |q|^2 \right)p   +   \left(\pZ (p\cdot q) - \qZ|p|^2 \right) q}{|\left( \qZ (p\cdot q) - \pZ |q|^2 \right)p   +   \left(\pZ (p\cdot q) - \qZ|p|^2 \right) q|},
\qquad \lambda_3 = |\qZ p - \pZ q|^2.  \label{v3}
\end{align}
We will directly use these eigenvalues and eigenvectors to establish the second proof of Lemma \ref{sij.ident.lem}.

\begin{proof}[Second proof of Lemma \ref{sij.ident.lem}]

Eigenvectors $v_1, v_2, v_3$ form an orthonormal basis, so any vector $\xi$ can be represented as:
\begin{align*}
\xi =(\xi \cdot v_1)v_1 +  (\xi \cdot v_2) v_2 +  (\xi \cdot v_3) v_3.
\end{align*}
Therefore, 
\begin{align}\label{eigenvalues.expansion}
S^{ij} \xi_i \xi_j 
&= \lambda_1  (\xi \cdot v_1)^2 + \lambda_2  (\xi \cdot v_2)^2 + \lambda_3 (\xi \cdot v_3)^2 \nonumber \\
&= 
\qS \rM \left(\xi \cdot \frac{(p\times q)}{|p\times q|} \right)^2
+
| \pZ q - \qZ p |^2  \left(\xi \cdot v_3 \right)^2,
\end{align}
where we used that $\lambda_1 =0$ and $\lambda_2 = \tau \rho$ by \eqref{taurho.exp}.   Clearly \eqref{eigenvalues.expansion} is non-negative.
\end{proof}

\begin{remark}
Here we point out that the subtracted expression in \eqref{sij.ident} is not lower order.  In particular if we choose $\xi$ orthogonal to both $p$ and $q$ with $|\xi | =1$ then as in \eqref{pij} and \eqref{aij} we have that
$
S^{ij} \xi_i \xi_j = | \qZ p - \pZ q|^2 - | p \times q |^2.
$ 
Further for any small $\epsilon>0$ consider
$$
B \eqdef (1-\epsilon) | \qZ p - \pZ q|^2 - | p \times q |^2.
$$
We will find conditions where $B<0$.

Suppose that $|p| = |q|$ then 
$
B = (1-\epsilon) \left( 2(\pZ)^2 |p|^2 -2(\pZ)^2 |p|^2 \cos\theta \right) - | p |^4 \sin^2 \theta .
$
We can calculate that 
$
B =  |p|^2 A_2(\theta)  + | p |^4 A_1( \theta)
$
for $A_1$ and $A_2$ that do not depend upon $p$.  In particular, after some calculation, $A_1( \theta) = 2\sin^2(\theta/2)\left(1-\cos\theta-2\epsilon \right)$.
Then $A_1(\theta)<0$ if $1-2\epsilon < \cos\theta<1$, and $B<0$ if $|p|$ is large enough.
\end{remark}

Note that we effectively ignore the case when $p$ and $q$ are co-linear, since it is a measure zero set for fixed $p$ or $q$.  Even so, it is shown in \cite{MR1773932}, that when $p=\lambda q$ for some $\lambda\in \mathbb{R}$ then the Landau kernel \eqref{kernelnormalized} is a multiple of the non-relativistic kernel $\phi^{ij}$ from Remark \ref{landau.remark.classical}.

\subsection{Landau Operator in non-conservative form}\label{landau.non.conservative.form}
In this section, we will express the Landau operator \eqref{landauC} in non-conservative form.  We don't actually use these results in the rest of the paper but we think it is important to explain the complicated computation.    First, we expand the collision operator from \eqref{landauC} with kernel \eqref{kernelnormalized}, \eqref{lambdaL} and \eqref{lambdaS} where we use the Einstein summation convention
\begin{gather*}
\mathcal{C}( h,g)(p)
=
\partial_{p_i} \int_{\mathbb{R}^3}\Phi^{ij}(p,q)\left\{h(q) \partial_{p_j}  g(p)  - g(p) \partial_{q_j} h(q)  \right\}
dq
\\
=
\left( \int_{\mathbb{R}^3}\Phi^{ij}(p,q)h(q)dq \right) \partial_{p_i}\partial_{p_j}  g(p)
+
\partial_{p_i}\left( \int_{\mathbb{R}^3} \Phi^{ij}(p,q)h(q)dq \right)\partial_{p_j}  g(p)
\\
-
\left( \int_{\mathbb{R}^3}\Phi^{ij}(p,q) \partial_{q_j}  h(q)dq \right)\partial_{p_i}  g(p)
-
\partial_{p_i}\left( \int_{\mathbb{R}^3} \Phi^{ij}(p,q) \partial_{q_j}  h(q)dq \right) g(p).
\end{gather*}
Note that these are not exactly convolutions.  

Now we recall a result from \cite{MR2100057}:

\begin{lemma} \label{posIT}
Given a smooth scalar function $G(q)$ which decays rapidly at infinity, we have
\begin{eqnarray}
-\partial_{p_i}\int_{\mathbb{R}^3}\Phi^{ij}(p,q)
\partial_{q_j}G(q) dq 
&=&
4\int_{\mathbb{R}^3}
\frac{ (\rM +1) }{ \pZ  \qZ }\left\{ \qS \rM \right\}^{-1/2}
G(q) dq
\nonumber
\\
&&
~~~~~~~~~~~~~+
\kappa(p) G(p),
\nonumber
\end{eqnarray}
where 
\begin{equation}\label{kappaDEF}
\kappa(p)=2^{7/2}\pi  \pZ \int_{0}^\pi \left(1+|p|^2\sin^2\theta\right)^{-3/2}\sin\theta d\theta.
\end{equation}
\end{lemma}

For the last term we use Lemma \ref{posIT} to obtain
$$
-\partial_{p_i}\left( \int_{\mathbb{R}^3} \Phi^{ij}(p,q)\partial_{q_j} h(q)dq \right)
=
4\int_{\mathbb{R}^3}
\frac{ (\rM+1) }{ \pZ  \qZ }
\frac{h(q)}{\sqrt{\qS \rM }}
 dq
+
\kappa(p) h(p)
$$
with \eqref{kappaDEF}.  Notice further that 
\begin{eqnarray}
\notag
\left(\partial_{p_i}+\frac{ \qZ }{ \pZ } \partial_{q_i}\right)  \pM q_\mu 
&=&\left(\partial_{p_i}+\frac{ \qZ }{ \pZ } \partial_{q_i}\right)\left( \pZ  \qZ -p\cdot q\right)
\\
&=&
\frac{p_i}{ \pZ } \qZ -q_i+\frac{ \qZ }{ \pZ }\left(\frac{q_i}{ \qZ } \pZ -p_i\right)
=0.
\label{calculationTHETA}
\end{eqnarray}
This is a key observation from \cite{MR2100057} which allows us do analysis on the relativistic Landau collision operator.

For the terms where the derivative is on the kernel, terms such as \eqref{lambdaL} and \eqref{lambdaS}, we use \eqref{calculationTHETA} and the following operator
\begin{equation}\notag
\Theta_i \eqdef \left(\partial_{p_i}+\frac{ \qZ }{ \pZ } \partial_{q_i}\right).
\end{equation}
Then for the coefficient of the second term we have 
\begin{gather*}
\partial_{p_i}\left( \int_{\mathbb{R}^3} \Phi^{ij}(p,q)h(q)dq \right)
=
 \int_{\mathbb{R}^3} \Theta_i\Phi^{ij}(p,q)h(q)dq 
-
 \int_{\mathbb{R}^3} \frac{ \qZ }{ \pZ }\partial_{q_i}\Phi^{ij}(p,q) h(q) dq.
\end{gather*}
Notice that, using the notation from \eqref{phi.not},  we can interchange $i$ and $j$ as 
$$
(\Phi^{ij})(\partial_{q_j} h)(q)(\partial_{p_i}  g)(p)
=
(\Phi^{ij})(\partial_{q_i} h)(q)(\partial_{p_j}  g)(p),
$$
since the matrix $\Phi^{ij}$ is symmetric.

Then collecting the second and third terms together we have
\begin{multline*}
\partial_{p_i}\left( \int_{\mathbb{R}^3} \Phi^{ij}(p,q)h(q)dq \right) \partial_{p_j}  g(p)
-
\left( \int_{\mathbb{R}^3}\Phi^{ij}(p,q) \partial_{q_i}  h(q)dq \right)\partial_{p_j}  g(p)
\\
=
\left( \int_{\mathbb{R}^3} \Theta_i\Phi^{ij}(p,q)h(q)dq \right) \partial_{p_j}  g(p)
 \\
+
\left( \int_{\mathbb{R}^3} \left( 1-\frac{ \qZ }{ \pZ }\right)\partial_{q_i}\Phi^{ij}(p,q) h(q) dq \right) \partial_{p_j}  g(p).
\end{multline*}
The point of this decomposition is that, fortunately, these integrands, $\Theta_i\Phi^{ij}(p,q)$ and $\left( 1-\frac{ \qZ }{ \pZ }\right)\partial_{q_i}\Phi^{ij}(p,q)$, have the same order singularity as $\Phi^{ij}$ itself.

Then we define the following operators
\begin{equation}\label{aijhdef}
a^{ij}(h)
=
a^{ij}(h)(p)
\eqdef
\int_{\mathbb{R}^3}\Phi^{ij}(p,q)h(q)dq,
\end{equation}
\begin{equation}\label{bj1hdef}
b^j(h)
=
b^j(h)(p)
\eqdef
\int_{\mathbb{R}^3} \left( \Theta_i\Phi^{ij}(p,q) + \left( 1-\frac{ \qZ }{ \pZ }\right)\partial_{q_i}\Phi^{ij}(p,q) \right) h(q)dq 
\end{equation}
and
\begin{equation}\label{c1def}
c(h)
=
c(h)(p)
\eqdef
4\int_{\mathbb{R}^3}
\frac{ 1 }{ \pZ  \qZ }
\frac{\rM+1}{\sqrt{\rM \qS}}
 h(q) dq
+
\kappa(p) h(p).
\end{equation}

We will further simplify the expression in \eqref{bj1hdef}, regarding this expression we have directly that 
$$
\Theta_i\Phi^{ij}(p,q)  + \left( 1-\frac{ \qZ }{ \pZ }\right)\partial_{q_i}\Phi^{ij}(p,q) 
=
\left(\partial_{p_i}+ \partial_{q_i}\right)\Phi^{ij}(p,q).
$$
Notice from \eqref{intermediate2} and symmetry that 
\begin{equation}
 \partial_{q_i}\Phi^{ij}(p,q)
=
2 \Lambda(p,q) \left(  (\rM+1)   p_j- q_j\right),
\quad 
 \partial_{p_i}\Phi^{ij}(p,q)
=
2 \Lambda(p,q) \left(  (\rM+1)   q_j- p_j\right).
\notag
\end{equation}
Then from the previous two expressions we obtain
\begin{equation}
\left(\partial_{p_i}+ \partial_{q_i}\right)\Phi^{ij}(p,q)
=
2 \Lambda(p,q) \rM \left(   p_j+ q_j\right),
\notag
\end{equation}
which also has a first order singularity.    We conclude that \eqref{bj1hdef} can be written as 
\begin{equation}\label{bj2hdef}
b^j(h)
=
2\int_{\mathbb{R}^3}  \Lambda(p,q) \rM \left(   p_j+ q_j\right)  h(q)dq. 
\end{equation}
This is the main expression that we will use for $b^j$.

Now we express the Landau operator in non-conservative form as
\begin{gather}\label{landauNonConservative}
\mathcal{C}(h,g)(p)
=
a^{ij}(h) \partial_{p_i}\partial_{p_j}  g(p)
+
b^j(h)\partial_{p_j}  g(p)
+
c(h)g(p).
\end{gather}
Here we use \eqref{aijhdef}, \eqref{bj2hdef} and \eqref{c1def}.

In the next section we will prove the main entropy dissipation estimate.

\section{Entropy dissipation estimate}\label{sec.entropy.est}

Our goal in this section is to prove the Theorem \ref{entropy.thm} which grants the uniform lower bound on the entropy dissipation.  We use the strategy from \cite{MR3557719} and \cite{MR3369941}.  The main new difficulties are algebraic and have to do with the extremely complicated relativistic algebraic structure.

The key estimate in proving the entropy dissipation estimate is the following lower bound of the determinant $\Delta_\phi(f)$ defined below.
\begin{lemma}\label{determinant}
Let $f$ be a non-negative function in $L^1_1(\threed)$, and 
let $\phi$ be a radially symmetric function which decays sufficiently fast at infinity.
Assume $H(f) \leq \overline{H}$.
Then, for all $i,j \in \{ 1,2,3\}$, with $i \neq j$, we have
\begin{align*}
\Delta_\phi(f) 
&  \eqdef  \det  \left( \int_\threed \phi(\frac{|q|^2}{2}) f(q) \begin{bmatrix}
1& \qi & \qj \\
\qi & \left(\qi\right)^2 & \qij \\
\qj & \qij & \left(\qj\right)^2
\end{bmatrix}
\d q \right)\\
& \geq \e_4^6 \left( \frac{1}{4} \int_\threed f(q) \d q\right)^3 \left( \inf_{B(0,R)} \phi\left(\frac{|q|^2}{2}\right) \right)^3 ,
\end{align*}
where
\begin{align*}
& \e_4 \eqdef \inf \left\{ \frac{1}{2},\,\e_0(R), \, \e_1(R), \,\e_2(R),   \e_3(M_0(f), M_1(f), \overline{H})  \right\}>0,\\
& R \eqdef \sup \left( 1, \sqrt{16 \left( \frac{\int_\threed f(q) \qZ \d q}{\int_\threed f(q) \d q}\right)^2 -1}\right)\\
&\e_0(R)  \eqdef  \frac{1}{4} \left( 1 - \frac{R}{\sqrt{1+R^2}} \right),\\
&\e_1(R) \eqdef  \frac{1}{4} \left( \sqrt{\frac{1+ 4R^2}{3+4R^2}} - \frac{R}{\sqrt{1+R^2}} \right),\\
&\e_2(R) \eqdef \frac{1 - \sqrt{\frac{1+4R^2}{2+4R^2}}}{\sqrt{2}},
\end{align*}
\begin{align*}
&\e_3(M_0(f), M_1(f), \overline{H})  \eqdef  \\
&\quad \frac{1}{4240} \left[ \sup \left( 1, \sqrt{\left( \frac{\int_\threed f(q) \qZ \d q}{\int_\threed f(q) \d q}\right)^2 -1}\right) \right]^{-6}
 \exp\left( \frac{-4\overline{H}}{\int_\threed f(q) \d q} \right)
\left(  \int_\threed f(q) \d q \right).
\end{align*}
\end{lemma}

\begin{proof}[Proof of Lemma \ref{determinant}]
Define 
\begin{align*}
B \eqdef  \begin{bmatrix}
1& \qi & \qj \\
\qi & \left(\qi\right)^2 & \qij \\
\qj & \qij & \left(\qj\right)^2
\end{bmatrix}.
\end{align*}
Then,
\begin{align}\label{Delta.abbreviated}
\Delta_\phi(f) = \det G,
\end{align}
where $G$ is the following matrix
\begin{align}\label{matrixG}
G= \int_\threed   \phi(\frac{|q|^2}{2}) f(q)   \, B \, \d q.
\end{align}

Since $G$ is a symmetric and real-valued matrix, it is diagonizable by an orthonormal matrix $O$ so that
\begin{align*}
O^T G O = 
\begin{bmatrix}
e_1 & 0 & 0 \\
0 & e_2 & 0 \\
0 & 0 & e_3
\end{bmatrix},
\end{align*}
where the orthonormal matrix $O$ can be represented as
\begin{align*}
O =
\begin{bmatrix}
\lambda_1 & \lambda_2 & \lambda_3 \\
\mu_1 & \mu_2 & \mu_3 \\
\nu_1 & \nu_2 & \nu_3
\end{bmatrix},
\end{align*}
where $\lambda_i^2 + \mu_i^2 +\nu_i^2 =1$, for $i=1,2,3$.

Therefore,
\begin{align*}
O^TGO = \begin{bmatrix}
e_1 & 0 & 0 \\
0 & e_2 & 0 \\
0 & 0 & e_3
\end{bmatrix} 
&= O^T \left(\int_\threed    \phi(\frac{|q|^2}{2}) f(q)   \, B \, \d q \right) O \\
& = \int_\threed   \phi(\frac{|q|^2}{2}) f(q)   \, O^T B O\, \d q.
\end{align*}

Therefore,
\begin{align}\label{G.as.Bkk}
\det G & = \det( O^TGO ) \nonumber \\
 &= e_1 \, e_2 \, e_3 \nonumber \\
& = \prod_{k=1}^3  \int_\threed   \phi\left(\frac{|q|^2}{2}\right) f(q)  \, (O^T B O)_{kk}\, \d q.
\end{align}

Notice that the matrix $B$ can be represented as
\begin{align*}
B =
\begin{bmatrix}
1 \\ \qi \\ \qj
\end{bmatrix}
\begin{bmatrix}
1 & \qi & \qj
\end{bmatrix},
\end{align*}
which implies
\begin{align*}
& O^TBO   \\
& \,=   
\begin{bmatrix}
\lambda_1 & \mu_1& \nu_1 \\
\lambda_2 & \mu_2& \nu_2 \\
\lambda_3 & \mu_3& \nu_3
\end{bmatrix}
\begin{bmatrix}
1 \\ \qi \\ \qj
\end{bmatrix}
\begin{bmatrix}
1 & \qi & \qj
\end{bmatrix}
\begin{bmatrix}
\lambda_1 & \lambda_2 & \lambda_3 \\
\mu_1 & \mu_2 & \mu_3 \\
\nu_1 & \nu_2 & \nu_3
\end{bmatrix} \\
& \, = 
\begin{bmatrix}
\lambda_1 & \mu_1\qi & \nu_1\qj \\
\lambda_2 & \mu_2\qi & \nu_2\qj \\
\lambda_3 & \mu_3\qi & \nu_3\qj
\end{bmatrix}
\begin{bmatrix}
\lambda_1 + \mu_1\qi + \nu_1\qj &
\lambda_2 + \mu_2\qi + \nu_2 \qj &
\lambda_3 + \mu_3\qi +  \nu_3\qj
\end{bmatrix},
\end{align*}
and so,
\begin{align*}
(O^T B O)_{kk} = \left( \lambda_k + \mu_k\qi + \nu_k\qj \right)^2.
\end{align*}

Recalling \eqref{G.as.Bkk}, we now have
\begin{align}\label{detG}
\det G &= \prod_{k=1}^3  \int_\threed   \phi(\frac{|q|^2}{2}) f(q)  \,\left( \lambda_k + \mu_k\qi + \nu_k\qj \right)^2 \, \d q \nonumber \\
& \geq \left( \inf_{\{\lambda^2 + \mu^2 +\nu^2 =1\}}   \int_\threed  \phi(\frac{|q|^2}{2}) f(q)  \,\left( \lambda + \mu \qi + \nu \qj \right)^2 \, \d q  \right)^3.
\end{align}

Therefore, for any $R>0$ and $\e \in (0,\frac{1}{2})$
\begin{align}\label{D1}
\Delta_\phi(f) & 
\geq  \left( \inf_{\{\lambda^2 + \mu^2 +\nu^2 =1\}}   \int_\threed  \phi\left(\frac{|q|^2}{2}\right) f(q)   \,\left( \lambda + \mu \qi + \nu \qj \right)^2 \, \d q  \right)^3 \nonumber \\
& \geq \e^6 \left(   \inf_{\{\lambda^2 + \mu^2 +\nu^2 =1\}}   \int_\threed  \phi\left(\frac{|q|^2}{2}\right) f(q) \,  \chi_{\left| \lambda + \mu \qi + \nu \qj \right| \geq \e} \, \d q   \right)^3 \nonumber\\
&\geq \
\e^6 \left( \inf_{B(0,R)} \phi\left(\frac{|q|^2}{2}\right) \right)^3    
\nonumber\\
& \qquad    
\cdot \left( \int_{B(0,R)} f(q) \d q \, - \, \sup_{\{\lambda^2 + \mu^2 +\nu^2 =1\}} \int_{B(0,R)} f(q) \,  \chi_{\left| \lambda + \mu \qi + \nu \qj \right| < \e} \, \d q    \right)^3.
\end{align}
The first of the two integrals can be estimated as follows
\begin{align}\label{int1}
 \int_{B(0,R)} f(q) \d q  
 \geq \int_\threed f(q) \d q - \frac{1}{\sqrt{1+R^2}}\int_\threed f(q) \qZ \d q. 
\end{align}
For the second integral, fix any $A>0$ and split the domain into two regions - where $|f| >A$ and where $|f|\leq A$. Using the fact that the entropy is bounded by $\overline{H}$ on the former domain, we have 
\begin{multline}\label{int2}
 \int_{B(0,R)} f(q) \,  \chi_{\left| \lambda + \mu \qi + \nu \qj \right| < \e} \, \d q  
  \\
\leq \frac{\overline{H}}{\ln A} \, + \, A  \left| \left\{  \left| \lambda + \mu \qi + \nu \qj \right| < \e  \right\}  \cap B(0,R)\right|.
\end{multline}
Combining \eqref{int1} and \eqref{int2} with \eqref{D1}, we have that for any $A>1$, 
\begin{align}\label{Delta.via.Y}
\Delta_\phi(f) & \geq \e^6 \left( \inf_{B(0,R)} \phi\left(\frac{|q|^2}{2}\right) \right)^3  \cdot 
\left(  \int_\threed f(q) \d q - \frac{1}{\sqrt{1+R^2}}\int_\threed f(q) \qZ \d q \right.  \\
& \left. \qquad \qquad \qquad \qquad \qquad \, - \, \frac{\overline{H}}{\ln A} \, - \, A \sup_{\{\lambda^2 + \mu^2 +\nu^2 =1\}} Y_{\{\lambda, \mu, \nu, R, \e\}}    \right)^3, \nonumber
\end{align}
where 
\begin{align}\label{Y1}
Y_{\{\lambda, \mu, \nu, R, \e\}} & = \left| \left\{ q\in \threed : \left| \lambda + \mu \qi + \nu \qj \right| < \e  \right\}  \cap B(0,R)\right| \nonumber \\
& = \int_{B(0,R)} \chi_{\left| \lambda + \mu \qi + \nu \qj \right| < \e} \, \d q \nonumber \\
& = \int_{B(0,R)} \chi_{\left| \lambda + \mu \2 + \nu \3 \right| < \e} \, \d q,
\end{align}
where the last equality exploits the fact that $i \neq j$, and is obtained by renaming the variables 
$(q_k, q_i, q_j) \mapsto (q_1, q_2, q_3)$.  Here $\{ k\} = \{1,2,3\} \setminus \{i, j\}$.

Next, by rotating the coordinate system, one can show that 
\begin{align}\label{Y2}
Y_{\{\lambda, \mu, \nu, R, \e\}} 
& =  \int_{B(0,R)} \chi_{\left| \tilde{\lambda} + \tilde{\mu} \frac{q_2}{\qZ} \right| < \e} \, \d q,
\end{align}
for some $\tilde{\lambda}$ and $\tilde{\mu}$ that satisfy $\tilde{\lambda}^2 + \tilde{\mu}^2 =1$. Indeed, consider the following rotation matrix
\begin{align}\label{O_1}
O_1  \eqdef  
\begin{pmatrix}
1 & 0 & 0\\
0& \frac{-\mu}{\sqrt{\mu^2 + \nu^2}} & \frac{-\nu}{\sqrt{\mu^2 + \nu^2}} \\
0& \frac{-\nu}{\sqrt{\mu^2 + \nu^2}} & \frac{\mu}{\sqrt{\mu^2 + \nu^2}}
\end{pmatrix}.
\end{align}
Note that $O_1$ is symmetric, real-valued and orthogonal matrix, so
\begin{align*}
&O_1=O_1^T=O_1^{-1},\\
&O_1 O_1 = I.
\end{align*}
Also note that
\begin{align*}
\begin{pmatrix}
\lambda &\mu & \nu
\end{pmatrix}
O_1 = 
\begin{pmatrix}
\tilde{\lambda} &\tilde{\mu} & 0
\end{pmatrix},
\end{align*}
where
\begin{align}\label{new.lm}
&\tilde{\lambda} = \lambda, \\
& \tilde{\mu} = -\sqrt{\mu^2 + \nu^2}, \nonumber \\
&\tilde{\lambda}^2 + \tilde{\mu}^2 = \lambda^2 + \mu^2 +\nu^2 =1. \nonumber
\end{align}
 Recalling that $O_1 O_1 = I$, from \eqref{Y1}, we have
\begin{align*}
Y_{\{\lambda, \mu, \nu, R, \e\}} 
&=
 \int_{B(0,R)}\chi_{\left| 
\begin{pmatrix}
\lambda & \mu & \nu
\end{pmatrix}
 O_1 O_1
\begin{pmatrix}
1 \\ \2 \\ \3
\end{pmatrix} \right| < \e} \, \d q \\
& =
 \int_{B(0,R)}\chi_{\left| 
\begin{pmatrix}
\tilde{\lambda} & \tilde{\mu} & 0
\end{pmatrix}
\begin{pmatrix}
1 \\ 
\frac{-\mu q_2 - \nu q_3}{\qZ \sqrt{\mu^2+\nu^2}} \\ 
\frac{-\nu q_2 + \mu q_3}{\qZ \sqrt{\mu^2+\nu^2}}
\end{pmatrix} 
\right| < \e} \, \d q.
\end{align*}
Now apply the change of variables $p= O_1 q,$
that is,
\begin{align*}
p_1&=q_1\\
p_2 &= \frac{-\mu q_2 -\nu q_3}{\sqrt{\mu^2+\nu^2}}\\
p_3 &=  \frac{-\nu q_2 +\mu q_3}{\sqrt{\mu^2+\nu^2}},
\end{align*}
and note that $\pZ = \qZ$, $\d p = \d q$, to conclude
\begin{align*}
Y_{\{\lambda, \mu, \nu, R, \e\}} 
=
 \int_{B(0,R)}\chi_{\left| 
\begin{pmatrix}
\tilde{\lambda} & \tilde{\mu} & 0
\end{pmatrix}
\begin{pmatrix}
1 \\ 
\frac{p_2}{\pZ} \\ 
\frac{p_3}{\pZ}
\end{pmatrix} 
\right| < \e} \, \d p
=
\int_{B(0,R)} \chi_{\left| \tilde{\lambda} + \tilde{\mu} \frac{q_2}{\qZ} \right| < \e} \, \d q
\end{align*}
which proves \eqref{Y2}. Since the parameter $\nu$ no longer plays a role, we introduce the following notation:
\begin{align}\label{Y^0}
Y^{0}_{\{\lambda, \mu, R, \e\}}  \eqdef  \int_{B(0,R)} \chi_{\left| \lambda + \mu \2 \right| < \e} \, \d q,
\end{align}
where $R>0$, $\lambda^2 + \mu^2 =1$ and $\e \in (0, \frac{1}{2})$.

We now proceed to estimate $Y^{0}_{\{\lambda, \mu, R, \e\}}$, for $R>0$, $\lambda^2 + \mu^2 =1$ and 
$\e \in (0, \frac{1}{2})$. 
\begin{itemize}
\item[(a)] { Case:} $ |\mu| \leq \frac{1}{2}(1-\e) $.  Note that for such  $\mu$ we have
\begin{align*}
\left| \lambda + \mu \qi \right| 
\geq |\lambda| - |\mu| \
= \sqrt{1-\mu^2} - |\mu|
 \geq 1 - 2|\mu| 
\geq \e,
\end{align*}
 Therefore,  if  $|\mu|\leq \frac{1}{2}(1-\e)$, the set $\left\{q:  \left| \lambda + \mu \2 \right| < \e \right\}$ is empty, that is, 
$Y^0_{\{\lambda, \mu, R, \e\}} =0$. So, without the loss of generality, from now on we assume that 
\begin{align}\label{empty1}
|\mu| \geq   \frac{1}{2}(1-\e)>\frac{1}{4}.
\end{align}

\item[(b)] { Case:} $\frac{1}{4} < |\mu| \leq \frac{1}{\sqrt{2}}.$     
We will now show the set 
$\left\{q:  \left| \lambda + \mu \2 \right| < \e \right\}   \cap B(0,R)$ 
is empty for sufficiently small $\e$.
 First note that $|\lambda + \mu \2| <\e$ is equivalent to 
\begin{align}\label{q2-a}
\frac{-\e}{|\mu|} \pm \sqrt{\frac{1}{\mu^2} -1} 
< \2
< \frac{\e}{|\mu|} \pm \sqrt{\frac{1}{\mu^2} -1}.
\end{align}
The function $q_i \mapsto \qi$ is increasing, since its derivative in $q_i$ is 
\begin{align*}
\frac{\qZ - q_i \frac{q_i}{\qZ}}{(\qZ)^2} 
= \frac{1+\sum_{j\ne i} q_j^2}{(\qZ)^3} >0.
\end{align*}
Therefore, for $q\in B(0,R)$ we have
\begin{align}\label{q2-b}
\frac{-R}{\sqrt{1+R^2}} 
\leq \frac{-R}{\sqrt{1+q_1^2+ R^2 +q_3^2}}
\leq \2
\leq \frac{R}{\sqrt{1+ q_1^2 + R^2 +q_3^2}}
\leq \frac{R}{\sqrt{1+R^2}}.
\end{align}
For $q$ to satisfy both \eqref{q2-a} and \eqref{q2-b}, we need to have (regardless of the sign in \eqref{q2-a})
\begin{align}\label{q2-c}
 \frac{-\e}{|\mu|} + \sqrt{\frac{1}{\mu^2}-1} < \frac{R}{\sqrt{1+R^2}}.
\end{align}

However, if we define 
\begin{align}\label{e0}
\e_0(R)  \eqdef  \frac{1}{4} \left( 1 - \frac{R}{\sqrt{1+R^2}} \right)>0,
\end{align}
then for any $\e \in (0, \e_0(R))$ 
we have
\begin{align*}
4\e < 1 - \frac{R}{\sqrt{1+R^2}} \leq  \sqrt{\frac{1}{\mu^2}-1} - \frac{R}{\sqrt{1+R^2}},
\end{align*}
where in the last inequality we used that  $|\mu| \leq \frac{1}{\sqrt{2}}$. Therefore,
\begin{align*}
\frac{R}{\sqrt{1+R^2}} < -4\e + \sqrt{\frac{1}{\mu^2}-1} <  \frac{-\e}{|\mu|} + \sqrt{\frac{1}{\mu^2}-1},
\end{align*}
due to \eqref{empty1}. This contradicts \eqref{q2-c}. 
Therefore, if $|\mu| \leq \frac{1}{\sqrt{2}}$ and $\e \in (0,\e_0(R))$,
then  $Y^0_{\{\lambda, \mu, R, \e\}} =0$ since the set $\left\{q:  \left| \lambda + \mu \2 \right| \leq \e \right\} \cap B(0,R)$ 
is empty. So, from now on we can assume that
\begin{align}\label{empty2}
|\mu| > \frac{1}{\sqrt{2}}.
\end{align}

\item[(c)] Case $ \frac{1}{\sqrt{2}} < |\mu| \leq  \sqrt{\frac{1+\delta}{2}}$, where 
\begin{align}\label{delta}
\delta = \delta(R): = \frac{1}{2(1+2R^2)}.
\end{align}
Again, we will show that for $\e$ small enough  $Y^0_{\{\lambda, \mu, R, \e\}} =0$. Set 
\begin{align}\label{e1}
\e_1(R) \eqdef  \frac{1}{4} \left( \sqrt{\frac{1-\delta}{1+\delta}} - \frac{R}{\sqrt{1+R^2}} \right).
\end{align}
 Note that  $\e_1(R) >0$. Indeed, 
\begin{align*}
\sqrt{\frac{1-\delta}{1+\delta}} - \frac{R}{\sqrt{1+R^2}}
& =
\sqrt{\frac{4R^2 +1}{4R^2 +3}} - \frac{R}{\sqrt{1+R^2}}\\
& >
\sqrt{\frac{4R^2 }{4R^2 +4}} - \frac{R}{\sqrt{1+R^2}} =0.
\end{align*}

Let $\e \in (0, \e_1(R)).$  Then,
\begin{align*}
4\e < \sqrt{ \frac{1-\delta}{1+\delta}} - \frac{R}{\sqrt{1+R^2}} \leq  \sqrt{\frac{1}{\mu^2}-1} - \frac{R}{\sqrt{1+R^2}},
\end{align*}
where in the second inequality we used that $ |\mu| \leq  \sqrt{\frac{1+\delta}{2}}$. As before, this implies
\begin{align*}
\frac{R}{\sqrt{1+R^2}} < -4\e + \sqrt{\frac{1}{\mu^2}-1} <  \frac{-\e}{|\mu|} + \sqrt{\frac{1}{\mu^2}-1},
\end{align*}
since $|\mu| > \frac{1}{\sqrt{2}}>\frac{1}{4}$. This again contradicts \eqref{q2-c}, which implies that the set
$\left\{q:  \left| \lambda + \mu \2 \right| \leq \e \right\} \cap B(0,R)$  is empty.  Thus, from now on we can assume that
\begin{align}\label{empty3}
|\mu| >  \sqrt{\frac{1+\delta}{2}}.
\end{align}

Note that  \eqref{empty3}  implies
\begin{align}\label{a}
a \eqdef   \sqrt{\frac{1}{\mu^2}-1} \in \left(0, \sqrt{\frac{1-\delta}{1+\delta}}\right).
\end{align}

\end{itemize}

We are now ready to estimate the size of the set
$\left\{q:  \left| \lambda + \mu \2 \right| \leq \e \right\} \cap B(0,R)$.
The set $\left\{q:  \left| \lambda + \mu \2 \right| < \e \right\}$ is equivalent to 
\begin{align*}
\left\{q \, : \,    \frac{-\e}{|\mu|} \pm \sqrt{\frac{1}{\mu^2}-1} <  \2<   \frac{\e}{|\mu|} \pm \sqrt{\frac{1}{\mu^2}-1}  \right\}, 
\end{align*}
where the sign $+$ corresponds to the case $\lambda <0$, and sign $-$ corresponds to $\lambda>0$. 
Let 
\begin{align}\label{e2}
\e_2(R): = \frac{1 - \sqrt{1-\delta}}{\sqrt{2}}.
\end{align}
Then for $\e \in (0, \e_2(R))$, and $\mu$ that satisfies \eqref{empty3}, we have
\begin{align}\label{<1}
\left| \pm \frac{\e}{|\mu|} \pm \sqrt{\frac{1}{\mu^2}-1}\right| 
& \leq 
\frac{\e}{|\mu|} + \sqrt{\frac{1}{\mu^2}-1} \nonumber \\
 & <
\frac{1 - \sqrt{1-\delta}}{\sqrt{2}} \frac{\sqrt{2}}{\sqrt{1+\delta}}  +   \sqrt{\frac{1-\delta}{1+\delta}}  \nonumber \\
& =
\frac{1}{\sqrt{1+\delta}} <1.
\end{align}

Let us also introduce the following notation:
\begin{align}\label{ul}
& c_1 \eqdef  - \tilde{\e} \pm a = -  \frac{\e}{|\mu|} \pm \sqrt{\frac{1}{\mu^2}-1}, \\
& c_2 \eqdef  \tilde{\e} \pm a =  \frac{\e}{|\mu|} \pm \sqrt{\frac{1}{\mu^2}-1},
\end{align}
where $\tilde{\e} =  \frac{\e}{|\mu|}$. The estimate \eqref{<1} implies that
\begin{align*}
|c_1|  < \frac{1}{\sqrt{1+\delta}} <1 \quad  \mbox{and} \quad |c_2|  <\frac{1}{\sqrt{1+\delta}} <1,
\end{align*}
which ensures that all the expressions below are well-defined.\\
The set $\left\{q:  \left| \lambda + \mu \2 \right| < \e \right\}$
 lies between the following two surfaces:
\begin{align*}
& \2 = c_1, \\
& \2 = c_2.
\end{align*}
Notice that for $c\in (-1,1)$
\begin{align*}
\2 = c 
& \quad \ \Longleftrightarrow  \quad q_2^2 = c^2 (1+q_1^2 + q_2^2 + q_3^2)\\
& \quad \ \Longleftrightarrow  \quad q_2^2(1-c^2) = c^2  (1 + q_1^2 + q_3^2) \\
& \quad \ \Longleftrightarrow  \quad q_2^2 = \frac{c^2}{1-c^2}  (1 + q_1^2 + q_3^2) \\
& \quad \ \Longleftrightarrow  \quad q_2 = \frac{c}{\sqrt{1-c^2}} \sqrt{ 1 + q_1^2 + q_3^2},
\end{align*}
where in the last line we use the fact that $q_2$ and $c$ need to have the same sign (since $q_0$ is positive).
Therefore, the set $\{q: c_1 < \2 < c_2 \} \cap B(0,R)$ can also be obtained by rotation around the $y$-axis of the region between the following two curves:
\begin{align*}
& l(x) \eqdef \frac{c_1}{\sqrt{1-c_1^2}}\sqrt{1+x^2} \\
& u(x) \eqdef \frac{c_2}{\sqrt{1-c_2^2}}\sqrt{1+x^2}.
\end{align*}
Hence, its volume can be calculated as follows:
\begin{align}\label{Y0.int}
Y^0_{\{\lambda, \mu, R, \e\}} 
& = 2\pi \int_0^R x \left(u(x) - l(x) \right) \d x  \nonumber \\
& = 2\pi \left(  \frac{c_2}{\sqrt{1-c_2^2}} - \frac{c_1}{\sqrt{1-c_1^2}} \right) \int_0^R x \sqrt{1+x^2} \d x \nonumber\\
& = 2\pi \left(  \frac{c_2}{\sqrt{1-c_2^2}} - \frac{c_1}{\sqrt{1-c_1^2}} \right)  \frac{1}{3} \left( (1+R^2)^{\frac{3}{2}} -1 \right).
\end{align}
Note that
%
%
\begin{align*}
&
 \frac{c_2}{\sqrt{1-c_2^2}} - \frac{c_1}{\sqrt{1-c_1^2}}    
= \frac{\E \pm a}{\sqrt{1-c_2^2}} - \frac{-\E \pm a}{\sqrt{1-c_1^2}} 
\\
& = \frac{(\E \pm a)^2 - (-\E \pm a)^2}{\sqrt{1-c_2^2} \sqrt{1-c_1^2} \left((\E \pm a)\sqrt{1-(-\E \pm a)^2} + (-\E \pm a)\sqrt{1-(\E \pm a)^2}  \right)}
\\
&  = \frac{\pm 4a\E}{\sqrt{1-c_2^2} \sqrt{1-c_1^2} \left((\E \pm a)\sqrt{1-(-\E \pm a)^2} + (-\E \pm a)\sqrt{1-(\E \pm a)^2}  \right)}
\\
&  = \frac{4a\E}{\sqrt{1-c_2^2} \sqrt{1-c_1^2} \left((a\pm \E)\sqrt{1-(-\E \pm a)^2} + (a \mp \E)\sqrt{1-(\E \pm a)^2}  \right)},
\end{align*}

It is easy to check that
\begin{align}\label{denom12}
\sqrt{1-(\E \pm a)^2} &\geq \sqrt{1-(a+\E)^2}\\
\sqrt{1-(-\E \pm a)^2} &\geq \sqrt{1-(a+\E)^2}
\notag
\end{align}

Also, we have
\begin{align}\label{denom3}
(a\pm \E)\sqrt{1-(-\E \pm a)^2} + (a \mp \E)\sqrt{1-(\E \pm a)^2}  \geq 2 a  \sqrt{1-(a+\E)^2}.
\end{align}
Namely, \eqref{denom3} is in fact stating two inequalities:
\begin{align*}
(a + \E)\sqrt{1-(a-\E)^2} + (a - \E)\sqrt{1-(a+\E)^2}  & \geq 2 a  \sqrt{1-(a+\E)^2},\\
(a - \E)\sqrt{1-(a+\E)^2} + (a + \E)\sqrt{1-(a-\E)^2}  & \geq 2 a  \sqrt{1-(a+\E)^2}.
\end{align*}
Both of them are true thanks to \eqref{denom12}.

The estimates \eqref{denom12} - \eqref{denom3} imply 
\begin{align*}
\frac{c_2}{\sqrt{1-c_2^2}} - \frac{c_1}{\sqrt{1-c_1^2}} 
& \leq  
2\E \frac{1}{\left(1-(a+\E)^2\right)^{\frac{3}{2}}}.
\end{align*}

By \eqref{<1} we know that $a+\E < \frac{1}{\sqrt{1+\delta}} <1$, and by \eqref{empty3} we have that 
$\E = \frac{\e}{|\mu |} < \frac{\e \sqrt{2}}{\sqrt{1+\delta}}$. Therefore,
\begin{align*}
\frac{c_2}{\sqrt{1-c_2^2}} - \frac{c_1}{\sqrt{1-c_1^2}} 
& \leq  
\e \, \frac{2 \sqrt{2}}{\sqrt{1+\delta}}
 \left( \frac{1+\delta}{\delta} \right)^{\frac{3}{2}}
=
 2 \e \, \sqrt{2} \,
  \frac{(1+\delta)}{\delta^{3/2}}.
\end{align*}

Recalling from \eqref{delta} that 
$$
\delta = \frac{1}{2+4R^2} < \frac{1}{2},
$$
we see that
\begin{align*}
\frac{c_2}{\sqrt{1-c_2^2}} - \frac{c_1}{\sqrt{1-c_1^2}} 
\leq 
 3 \e \, \sqrt{2} \,
  \frac{1}{\delta^{3/2}}
=  
3 \e \, \sqrt{2} \, (2 + 4R^2)^{3/2}.
\end{align*}
Since $R$ will be chosen so that it is greater than $1$, we further have
\begin{align*}
\frac{c_2}{\sqrt{1-c_2^2}} - \frac{c_1}{\sqrt{1-c_1^2}} 
\leq 
3 \e \, \sqrt{2} \, (6R^2)^{3/2}
\leq 63 \, \e \, R^3.
\end{align*}

Therefore,  from \eqref{Y0.int} we have
\begin{align}
Y^0_{\{\lambda, \mu, R, \e\}} 
\leq   42 \pi \e R^3 \left( (1+R^2)^{\frac{3}{2}} -1 \right)  
\leq  1060 \e R^6,
\end{align}
where in the last inequality we use that  $(1+R^2)^{3/2} \leq 8R^3$ for $R\geq1.$
Finally, note that 
\begin{align*}
\sup_{\{\lambda^2 + \mu^2 +\nu^2 =1\}} Y_{\{\lambda, \mu, \nu, R, \e\}}  
=
\sup_{\{\lambda^2 + \mu^2  =1\}} Y^0_{\{\lambda, \mu, R, \e\}}.
\end{align*}
Therefore, for $R\geq 1$ and  
$$
0<\e \leq \inf\{\frac{1}{2},\e_0(R), \e_1(R), \e_2(R)\}, 
$$ 
with $\e_0(R), \e_1(R), \e_2(R)$
defined as in \eqref{e0}, \eqref{e1} and \eqref{e2},  we have
\begin{align}\label{Y.est}
\sup_{\{\lambda^2 + \mu^2 +\nu^2 =1\}} Y_{\{\lambda, \mu, \nu, R, \e\}}  
\leq 
 1060 \e R^6.
\end{align}
With this estimate  we can now find the bound for $\Delta_\phi(f)$. Namely, from  \eqref{Delta.via.Y}, we now have
\begin{align*}
\Delta_\phi(f) & \geq \e^6 \left( \inf_{B(0,R)} \phi\left(\frac{|q|^2}{2}\right) \right)^3  \cdot \\
& \qquad   \cdot  \nonumber
\left(  \int_\threed f(q) \d q - \frac{1}{\sqrt{1+R^2}}\int_\threed f(q) \qZ \d q 
\, - \, \frac{\overline{H}}{\ln A} \, - \,   1060 \e R^6 \right)^3.
\end{align*}

First, choose $R\geq 1$ so that 
\begin{align*}
 \frac{1}{\sqrt{1+R^2}}\int_\threed f(q) \qZ \d q \leq \frac{1}{4} \int_\threed f(q) \d q.
\end{align*}
In other words,
\begin{align}\label{R}
R \eqdef  \sup \left( 1, \sqrt{16\left( \frac{\int_\threed f(q) \qZ \d q}{\int_\threed f(q) \d q}\right)^2 -1}\right).
\end{align}

Then, choose $A$ so that
\begin{align*}
\frac{\overline{H}}{\ln A} =  \frac{1}{4} \int_\threed f(q) \d q.
\end{align*}
In other words,
\begin{align}\label{A}
A  \eqdef  \exp\left( \frac{4\overline{H}}{\int_\threed f(q) \d q} \right).
\end{align}
Finally, impose an additional condition on $\e$ so that
\begin{align*}
 1060 \e R^6 A \leq \frac{1}{4} \int_\threed f(q) \d q,
\end{align*}
that is,

\begin{align*}
\e \leq \e_3(M_0(f), M_1(f), \overline{H}),
\end{align*}
where
\begin{multline}\label{e3}
\e_3(M_0(f), M_1(f), \overline{H})  \eqdef  
\\
\frac{1}{4240} \left[ \sup \left( 1, \sqrt{\left( \frac{\int_\threed f(q) \qZ \d q}{\int_\threed f(q) \d q}\right)^2 -1}\right) \right]^{-6}
 \exp\left( \frac{-4\overline{H}}{\int_\threed f(q) \d q} \right)
\left(  \int_\threed f(q) \d q \right). 
\end{multline}

Gathering all conditions on $\e$, set
\begin{align}\label{eps}
\e_4  \eqdef  \inf \left\{ \frac{1}{2},\,\e_0(R), \, \e_1(R), \,\e_2(R),   \e_3(M_0(f), M_1(f), \overline{H})  \right\}>0.
\end{align}
With this choice of $R, A, \e$ we have the following estimate:
\begin{align*}
\Delta_\phi(f) & \geq \e_4^6 
\left( \frac{1}{4} \int_\threed f(q) \d q \right)^3
\left( \inf_{B(0,R)} \phi\left(\frac{|q|^2}{2}\right) \right)^3,
\end{align*}
which is the desired estimate.
\end{proof}

With the lower bound of $\Delta_\phi(f)$ at our disposal, we are now ready to prove Theorem \ref{entropy.thm}. The main idea is to consider a $3\times 3$ system of equations \eqref{three.integrals.tab} with three unknowns, one of which is $\frac{\partial_{p_i }f(p)}{f(p)}$. Then Cramer's rule will be used to express $\frac{\partial_{p_i} f(p)}{f(p)}$. As a result the inverse of $\Delta_\phi(f)$ will show up, and that is how Lemma \ref{determinant} will be used at the very end of the proof. But before this system of 3 equations is set up, we first find an auxiliary lower bound of the entropy dissipation \eqref{entropy.auxiliary} which exploits the fact that we know eigenvalues of the relativistic Landau operator. This auxiliary bound will be used again towards the end of the proof after an appropriate application of the Cauchy-Schwartz inequality.

\begin{proof}[Proof of Theorem \ref{entropy.thm}]
In order to control the kernel in the entropy dissipation we recall the expansion from \eqref{eigenvalues.expansion} and the expression \eqref{taurho.exp}.  With these, we have the following lower bound
\begin{multline}\label{kernel.lower1}
S^{ij} \xi_i \xi_j 
= 
\qS \rM \left(\xi \cdot \frac{(p\times q)}{|p\times q|} \right)^2
+
| \pZ q - \qZ p |^2  \left(\xi \cdot v_3 \right)^2
\\
\ge 
\qS \rM
\left(\left(\xi \cdot \frac{(p\times q)}{|p\times q|} \right)^2+  \left(\xi \cdot v_3 \right)^2\right)
\\
=
\qS \rM
\left(\left|\xi  \right|^2-  \left(\xi \cdot v_1 \right)^2\right)
=
\qS \rM
\left|v_1  \times \xi \right|^2
\\
=
\left( \left| q^0p -  p^0 q \right|^2 - \left| p \times q \right|^2 \right)
\frac{\left|\left( \qZ p - \pZ q \right)   \times \xi \right|^2}{\left| \qZ p - \pZ q \right|^2}.
\end{multline}
Here we used that the eigenvectors $\{v_1=\frac{(\qZ p - \pZ q )}{|\qZ p - \pZ q |}, v_2=\frac{(p\times q)}{|p\times q|}, v_3\}$ from \eqref{v1}-\eqref{v3} form an orthonormal basis for $\mathbb{R}^3$ (when $p$ and $q$ are not colinear).  We also used the inequality 
$$
\left| q^0p -  p^0 q \right|^2  
\ge 
\left| q^0p -  p^0 q \right|^2 - \left| p \times q \right|^2
=
\qS \rM.
$$
Then more generally we have the following lower bound for \eqref{kernelnormalized} with \eqref{lambdaL} and \eqref{lambdaS} as
\begin{equation}\label{phi.lower.ab}
\Phi^{ij} \xi_i \xi_j 
\ge
\Lambda \qS \rM
\frac{\left|\left( \qZ p - \pZ q \right)   \times \xi \right|^2}{\left| \qZ p - \pZ q \right|^2}
=
\frac{( \rM+1)^2}{\pZ \qZ}
\left(  \qS \rM \right)^{-1/2}
\frac{\left|\left( \qZ p - \pZ q \right)   \times \xi \right|^2}{\left| \qZ p - \pZ q \right|^2}.
\end{equation}
%
Then from \eqref{entropy.dissipation} and \eqref{phi.lower.ab}, we have the lower bound for the entropy dissipation
\begin{multline}\label{entropy.auxiliary}
  D(f)  
  \\
= \frac{1}{2}\int_{\mathbb{R}^3}  
 \int_{\mathbb{R}^3} f(p) f(q) \Phi^{ij}(p,q)
 \left( \frac{\partial_{p_j}f}{f}(p)  - \frac{\partial_{q_j}f}{f}(q)  \right)
 \left( \frac{\partial_{p_i}f}{f}(p)  - \frac{\partial_{q_i}f}{f}(q)  \right) dqdp
 \\ 
  \geq  
  \frac{1}{2}\int_{\mathbb{R}^3}  
 \int_{\mathbb{R}^3} f(p) f(q)
\frac{( \rM+1)^2}{\pZ \qZ}
\left(  \qS \rM \right)^{-1/2}
\frac{\left|\left( \frac{p}{\pZ} - \frac{q}{\qZ} \right)   \times  \left( \frac{\nabla_p f}{f}(p)  -  \frac{\nabla_q f}{f}(q)  \right) \right|^2}{\left| \frac{p}{\pZ} - \frac{q}{\qZ} \right|^2}
\\   
=   
\frac{1}{4}\int_{\mathbb{R}^3}
 \int_{\mathbb{R}^3} f(p) f(q)
\frac{( \rM+1)^2}{\pZ \qZ}
\left(  \qS \rM \right)^{-1/2}
\left| \frac{p}{\pZ} - \frac{q}{\qZ} \right|^{-2}
\sum_{i,j=1}^3 \left| q_{ij}(p,q) \right|^2 dqdp,
 \end{multline}
where due to the identity 
\begin{align}\label{cross.prod.id}
|x \times y|^2 = \frac{1}{2} \sum_{i,j=1}^3 \left(x_i y_j - x_j y_i \right)^2,
\end{align}
we have
\begin{align*}
 q_{ij}(p,q)  
&=
\left( \pii - \qi \right) \left( \frac{\partial_{p_j}f}{f}(p)  - \frac{\partial_{q_j}f}{f}(q) \right)
-
\left( \pj - \qj \right) \left( \frac{\partial_{p_i}f}{f}(p)  - \frac{\partial_{q_i}f}{f}(q) \right)\\
&= 
\left(\pii \frac{\partial_{p_j}f}{f}(p) - \pj \frac{\partial_{p_i}f}{f}(p)  \right)
+ \qj  \frac{\partial_{p_i}f}{f}(p) - \qi  \frac{\partial_{p_j}f}{f}(p)\\
&
\quad - \pii  \frac{\partial_{q_j}f}{f}(q)  +  \pj \frac{\partial_{q_i}f}{f}(q) 
+ \left(  \qi \frac{\partial_{q_j}f}{f}(q) -   \qj \frac{\partial_{q_i}f}{f}(q) \right).
\end{align*}

From here, the strategy for finding the lower bound of $D(f)$ is to consider the following three integrals as a system of equations:
\begin{align}
\notag
& \int q_{ij}(p,q) \,  \phi(\frac{|q|^2}{2}) \, f(q) \, dq, \\
\label{three.integrals.tab}
& \int q_{ij}(p,q)  \, \qi \,  \phi(\frac{|q|^2}{2}) \,  f(q) \, dq,\\
\notag
& \int q_{ij}(p,q) \, \qj \, \phi(\frac{|q|^2}{2}) \,  f(q) \, dq,
\end{align}
where $ \phi(\frac{|q|^2}{2})$ is a generic radially symmetric function that we will assume decays at infinity sufficiently fast. These quantities will lead to a $3$ by $3$ system of equations,  with the unknowns
$\pii \frac{\partial_{p_j}f}{f}(p) -  \pj \frac{\partial_{p_i}f}{f}(p) $, 
$\frac{\partial_{p_i}f}{f}(p)$ and 
$\frac{\partial_{p_j}f}{f}(p)$.
Cramer's rule will be then used  to express and estimate $\frac{\partial_{p_i}f}{f}(p) $ . We now expand the three integrals. 
Let us introduce the notation:
\begin{align*}
X_1 & \eqdef  \pii \frac{\partial_{p_j}f}{f}(p) -  \pj \frac{\partial_{p_i}f}{f}(p), \\
X_2 &  \eqdef \frac{\partial_{p_i}f}{f}(p), \\
X_3 & \eqdef \frac{\partial_{p_j}f}{f}(p).
\end{align*}
Then we have
\begin{align*}
 & \int q_{ij}(p,q) \,  \phi(\frac{|q|^2}{2}) \, f(q) \, dq 
= 
  X_1 \left( \int_\threed \phi(\frac{|q|^2}{2}) \, f(q) \, dq \right)\\
& \qquad +  X_2 \left( \int_\threed   \qj  \phi(\frac{|q|^2}{2}) \, f(q) \, dq \right) 
 		 -  X_3  \left( \int_\threed   \qi  \phi(\frac{|q|^2}{2}) \, f(q) \, dq \right) \\
& \qquad  + \pii \, \left( \int_\threed  q_j  \phi'(\frac{|q|^2}{2}) \, f(q) \, dq \right) \, 
		 - \pj \, \left( \int_\threed  q_i  \phi'(\frac{|q|^2}{2}) \, f(q) \, dq \right). \, 
\end{align*}

Also,
\begin{align*}
& \int q_{ij}(p,q)  \, \frac{q_i}{\qZ} \,  \phi(\frac{|q|^2}{2}) \,  f(q) \, dq 
 =  
X_1  \left( \int_\threed \qi \, \phi(\frac{|q|^2}{2}) \, f(q) \, dq \right) \\
& \quad+ X_2   \left( \int_\threed \qj \qi  \phi(\frac{|q|^2}{2}) \, f(q) \, dq \right) 
  -   X_3  \left( \int_\threed \left( \qi \right)^2  \phi(\frac{|q|^2}{2}) \, f(q) \, dq \right) \\
&  \quad+  
\pii  \int_\threed f(q) \left( \frac{q_i q_j}{\qZ} \phi'(\frac{|q|^2}{2})  -  \frac{q_i q_j}{(\qZ)^3}\phi(\frac{|q|^2}{2}) \right) \d q  \\
& \quad 
 - \pj \int_\threed f(q) \left( \frac{\phi(\frac{|q|^2}{2}) + (q_i)^2 \phi'(\frac{|q|^2}{2})}{\qZ}   -  \frac{(q_i)^2}{(\qZ)^3}\phi(\frac{|q|^2}{2}) \right) \d q  \\
& \quad 
+ \int_\threed \frac{q_j}{(\qZ)^2} \phi(\frac{|q|^2}{2}) f(q) dq.
\end{align*}

Finally,
\begin{align*}
 &  \int q_{ij}(p,q) \, \frac{q_j}{\qZ} \, \phi(\frac{|q|^2}{2}) \,  f(q) \, dq 
 =  
X_1  \left( \int_\threed \frac{q_j}{\qZ} \, \phi(\frac{|q|^2}{2}) \, f(q) \, dq \right) \\
& \quad + 
 X_2 \left( \int_\threed \left( \frac{q_j}{\qZ}\right)^2 \phi(\frac{|q|^2}{2}) \, f(q) \, dq \right) 
 -  X_3 \left( \int_\threed \frac{q_i}{\qZ} \frac{q_j}{\qZ}   \phi(\frac{|q|^2}{2}) \, f(q) \, dq \right)  \\
& \quad + 
\pii \int_\threed f(q) \left( \frac{\phi(\frac{|q|^2}{2}) + (q_j)^2 \phi'(\frac{|q|^2}{2})}{\qZ}   -  \frac{(q_j)^2}{(\qZ)^3}\phi(\frac{|q|^2}{2}) \right) \d q  \\
& \quad -
 \pj  \int_\threed f(q) \left( \frac{q_i q_j}{\qZ} \phi'(\frac{|q|^2}{2})  -  \frac{q_i q_j}{(\qZ)^3}\phi(\frac{|q|^2}{2}) \right) \d q  \\
& \quad  -
\int_\threed \frac{q_i}{(\qZ)^2} \phi(\frac{|q|^2}{2}) f(q) dq.
\end{align*}

Therefore, we have the following $3 \times 3$ system:
\begin{align*}
&  \left( \int_\threed \phi(\frac{|q|^2}{2}) \, f(q) \, dq \right) X_1
+ \left( \int_\threed \frac{q_j}{\qZ} \phi(\frac{|q|^2}{2}) \, f(q) \, dq \right) X_2 \\
& \qquad -  \left( \int_\threed \frac{q_i}{\qZ} \phi(\frac{|q|^2}{2}) \, f(q) \, dq \right) X_3
=
\int_\threed \phi(\frac{|q|^2}{2}) \, f(q) \,
\Big( q_{ij}  + P_1(f)\Big) \d q \\
& \\
& \left( \int_\threed \frac{q_i}{\qZ} \, \phi(\frac{|q|^2}{2}) \, f(q) \, dq \right) X_1
+  \left( \int_\threed \frac{q_j}{\qZ} \frac{q_i}{\qZ} \phi(\frac{|q|^2}{2}) \, f(q) \, dq \right) X_2 \\
& \quad -  \left( \int_\threed \left( \frac{q_i}{\qZ} \right)^2  \phi(\frac{|q|^2}{2}) \, f(q) \, dq \right)  X_3
= \int_\threed \phi(\frac{|q|^2}{2}) \, f(q) \,
\Big( q_{ij}   \frac{q_i}{\qZ}  + P_2(f)\Big) \d q \\
& \\
& \left( \int_\threed \frac{q_j}{\qZ} \, \phi(\frac{|q|^2}{2}) \, f(q) \, dq \right) X_1
+ \left( \int_\threed \left( \frac{q_j}{\qZ}\right)^2 \phi(\frac{|q|^2}{2}) \, f(q) \, dq \right)  X_2\\
& \quad
-  \left( \int_\threed \frac{q_i}{\qZ} \frac{q_j}{\qZ}   \phi(\frac{|q|^2}{2}) \, f(q) \, dq \right) X_3
= \int_\threed \phi(\frac{|q|^2}{2}) \, f(q) \,
\Big( q_{ij}  \frac{q_j}{\qZ}  + P_3(f)\Big)  \d q,
\end{align*}

where
\begin{align*}
 P_1(f)(p,q) 
&= 
- \pii \,  \frac{q_j \, \phi'}{\phi} 
+ \pj \, \frac{ q_i \, \phi'}{\phi},\\
 P_2(f)(p,q) 
& =
-\pii \,  \left( \frac{q_i q_j \,\phi'  }{\qZ \, \phi}-  \frac{q_i q_j}{(\qZ)^3} \right) 
+ \pj \,  \left( \frac{\phi + (q_i)^2 \phi'}{\qZ \, \phi}   -  \frac{(q_i)^2}{(\qZ)^3} \right)
- \frac{q_j}{(\qZ)^2}\\
 P_3(f)(p,q) 
& = 
 -\pii \, \left( \frac{\phi + (q_j)^2 \phi'}{\qZ \, \phi}   -  \frac{(q_j)^2}{(\qZ)^3} \right)
+ \pj \, \left( \frac{q_i q_j \,\phi'  }{\qZ \, \phi}-  \frac{q_i q_j}{(\qZ)^3} \right)
+ \frac{q_i}{(\qZ)^2}.
\end{align*}

Cramer's formula yields
\begin{align*}
&\frac{\partial_{p_i}f}{f}(p) 
= \Delta_\phi(f)^{-1} 
 \det  \left( \int_\threed \phi(\frac{|q|^2}{2}) f(q) \begin{bmatrix}
1& \qi & q_{ij}  + P_1(f) \\
\qi & \left(\qi\right)^2 &q_{ij}   \frac{q_i}{\qZ}  + P_2(f)\\
\qj & \qij & q_{ij}  \frac{q_j}{\qZ}  + P_3(f)
\end{bmatrix}
\d q \right). 
\end{align*}
Taking into account that all the elements in the first two columns can be bounded by 
$\int_\threed \phi(\frac{|q|^2}{2}) f(q) \d q$, we have
\begin{align} \label{cramer.formula}
\left| \frac{\partial_{p_i}f}{f}(p)  \right| 
&\leq 
 2 \Delta_\phi(f)^{-1} 
\left( \int_\threed \phi(\frac{|q|^2}{2}) f(q)  \d q \right)^2 \\
&\qquad \cdot
\left(
 \int_\threed \phi(\frac{|q|^2}{2}) f(q)
\Big(
|P_1(f)| + |P_2(f)| + |P_3(f)| +  3|q_{ij}|
\Big)
\d q
\right).
\notag
\end{align}
Since $\left|\pii \right| \leq 1$ etc., we have 
\begin{align*}
|\phi| \, \Big(|P_1(f)| + |P_2(f)| + |P_3(f)| \Big)
& \leq 
3|\phi'| \Big( |q_i| +|q_j| \Big)    +    8|\phi| \\
& \leq
3\sqrt{2} |q| |\phi'| + 8 |\phi|.
\end{align*}
Therefore,
\begin{align*}
\left| \frac{\partial_{p_i}f}{f}(p)  \right| 
&\leq 
2 \Delta_\phi(f)^{-1} 
\left( \int_\threed \phi(\frac{|q|^2}{2}) f(q)  \d q \right)^2 \\
&\quad \cdot
\left[ 
3\sqrt{2} \int_\threed f(q) \phi'(\frac{|q|^2}{2})  |q|  \d q  
+
8 \int_\threed f(q) \phi(\frac{|q|^2}{2})   \d q \right. \\
& \left.
 \hspace{1.9in} + 3 \int_\threed f(q) \phi(\frac{|q|^2}{2}) |q_{ij}|  \d q
 \right].
\end{align*}

By squaring this inequality and using that $(a+b+c)^2 \leq 3( a^2 + b^2 +c^2)$, we have
\begin{align*}
\left| \frac{\partial_{p_i}f}{f}(p)  \right|^2
&\leq 
4 \Delta_\phi(f)^{-2} 
\left( \int_\threed \phi(\frac{|q|^2}{2}) f(q)  \d q \right)^4 \\
&\quad \cdot
\left[ 
54 \left( \int_\threed f(q) \phi'(\frac{|q|^2}{2})  |q|  \d q \right)^2  
+
192 \left(\int_\threed f(q) \phi(\frac{|q|^2}{2})   \d q \right)^2\right. \\
& \left.
 \hspace{1.97in} + 27 \left( \int_\threed f(q) \phi(\frac{|q|^2}{2}) |q_{ij}|  \d q \right)^2
 \right].
\end{align*}

Integrating the last inequality against $f(p)$, and using the Cauchy-Schwartz inequality on the last term, we get
\begin{align*}
&\int_\threed f(p) \left|  \frac{\partial_{p_i}f}{f}(p)  \right|^2 \d p
 \leq 
4\, \Delta_\phi(f)^{-2} 
\left( \int_\threed \phi(\frac{|q|^2}{2}) f(q)  \d q \right)^4 \\
& \quad \cdot
\left\{
 \left( \int_\threed f(p)  \d p \right) \right. 
\left[ 54 \left( \int_\threed f(q) \phi'(\frac{|q|^2}{2})  |q|  \d q \right)^2  
+ 192 \left( \int_\threed f(q) \phi(\frac{|q|^2}{2})   \d q \right)^2 \right] \\
& \qquad \qquad \left.
+ 27 \int_\threed f(p) 
\left( \int_\threed f(q) |q_{ij}|^2 A \d q\right) 
\left( \int_\threed f(q) \h^2  A^{-1} \d q\right)
\d p
\right\},
\end{align*}
where we choose a as 
\begin{equation}\label{A.def.thing}
A = \frac{(\rho +1)^2}{\pZ \qZ} \left( \tau \rho \right)^{-1/2} \left| \frac{p}{\qZ} - \frac{q}{\pZ} \right|^{-2},
\end{equation}
so that we can recognize the right hand side of \eqref{entropy.auxiliary} to obtain
$$ 
\int_\threed \int_\threed f(q) |q_{ij}|^2 A \d q \d p \leq 4D(f).
$$

Then we have
\begin{align*}
&\int_\threed f(p) \left|  \frac{\partial_{p_i}f}{f}(p)  \right|^2 \d p
 \leq 
4 \, \Delta_\phi(f)^{-2} 
\left( \int_\threed \phi(\frac{|q|^2}{2}) f(q)  \d q \right)^4 \\
& \quad \cdot
\left\{
 \left( \int_\threed f(p)  \d p \right) \right. 
\left[ 54 \left( \int_\threed f(q) \phi'(\frac{|q|^2}{2})  |q|  \d q \right)^2  
+ 192 \left(\int_\threed f(q) \phi(\frac{|q|^2}{2})   \d q \right)^2 \right] \\
& \qquad \qquad \left.
+ 108 \, D(f) \sup_p 
\left( \int_\threed f(q) \h^2  A^{-1} \d q\right)
\right\}.
\end{align*}
Here we {\it claim} that   
\begin{equation}\label{claim.one.proof}
 \sup_{p\in\threed} 
\left( \int_\threed f(q) \h^2  A^{-1} \d q\right)
\lesssim 1.
\end{equation}
This claim will be established below after the proof.  
Finally observe that 
\begin{align*}
\int_\threed   \left| \nabla \sqrt{f(p)}\right|^2 \d p
=
\frac{1}{4} \sum_{i=1}^3 
\int_\threed f(p) \left| \frac{\partial_{p_i}f}{f}(p) \right|^2 \d p.
\end{align*}
Therefore, if the function $\phi$ is chosen so that all integrals involving $\phi$ are finite, then we will have that
\begin{align*}
\int_\threed   \left| \nabla \sqrt{f(p)}\right|^2 \d p
& \leq
C_1  \Delta_\phi(f)^{-2}  + C_2 D(f),
\end{align*}
and then Lemma \ref{determinant} can be used to conclude
\begin{align*}
\int_\threed   \left| \nabla \sqrt{f(p)}\right|^2 \d p
& \leq
C_1   + C_2 D(f).
\end{align*}
This completes the proof.
\end{proof}

Next we will prove the claim in \eqref{claim.one.proof}.  But first we briefly recall a useful inequality taken from Glassey \& Strauss \cite{MR1211782}:

\begin{proposition}\label{lorenzest} Let $p,q\in\mathbb{R}^3$  then
\begin{equation}
\label{sestimate}
\frac{|p-q|^2+|p\times q|^2}{2  \pZ   \qZ }
\le  
\rM
\le 
\frac{1}{2}|p-q|^2.
\end{equation}
\end{proposition}

We multiply and divide by $ \pZ  \qZ +p\cdot q+1$ to observe that
$$
\rM=\frac{|p-q|^2+|p\times q|^2}{ \pZ  \qZ +p\cdot q+1}.
$$
Then we note that $\pZ  \qZ +p\cdot q \ge 1$ and $p\cdot q+1 \le \pZ  \qZ$.  Plugging these into the above yields Proposition \ref{lorenzest}.

\begin{proof}[Proof of the claim in \eqref{claim.one.proof}]
We recall that
$
A = \frac{(\rho +1)^2}{\pZ \qZ} \left( \tau \rho \right)^{-1/2} \left| \frac{p}{\qZ} - \frac{q}{\pZ} \right|^{-2},
$
and then clearly
$$
A^{-1} \le  \frac{\pZ \qZ}{(\rho +1)^2} \left( \left( \rho+2\right) \rho \right)^{1/2}
\le 
\frac{\pZ \qZ}{(\rho +1)}.
$$
First we assume that $|p| \le 2|q|$.  Then on this set 
$$
A^{-1} 
\lesssim
( \qZ)^2.
$$
And this would be enough to establish the claim since we allow $\phi(|q|^2/2)$ to be rapidly decaying. 
On the set $|p| \le 1$ then the claim holds.
 Next on the set $|p| \ge 2|q|$ we have from \eqref{sestimate} since $|p-q| \ge |p| - |q| \ge \frac{1}{2} |p|$ that
$$
\rM \ge \frac{|p-q|^2}{2  \pZ   \qZ }
\ge \frac{|p|^2}{8  \pZ   \qZ }.
$$
  So further if $|p| \ge 2|q|$ and $|p| \ge 1$ then 
$
\rM 
\gtrsim
\frac{\pZ }{    \qZ }.
$
And on this region 
$$
A^{-1} 
\lesssim
( \qZ)^2.
$$
And the claim also holds here since we allow $\phi(|q|^2/2)$ to be rapidly decaying.
\end{proof}

\subsection{Estimates on the kernel $\Phi^{ij}(p,q)$}\label{kernel.est.section}
In this section we will prove estimates on the kernel $\Phi^{ij}(p,q)$ from \eqref{kernelnormalized} and we will further prove uniform upper and lower bounds on the matrix $a^{ij}(h)$ from \eqref{aijhdef}.

\begin{lemma}
For the kernel from \eqref{kernelnormalized} with \eqref{gDEFINITION} we have the uniform pointwise upper bound:
\begin{align}\label{part1}
\left|  \Phi^{ij}(p,q) \right| \lesssim  
\begin{cases}
\frac{\sqrt{\pZ \qZ}}{|p-q|},   &\mbox{for} \, \rho < \frac{1}{8}, \\
\frac{\pZ}{\qZ} + \frac{\qZ}{\pZ},  	 &\mbox{for} \, \rho \geq \frac{1}{8}.
\end{cases}
\end{align}
Further, recalling \eqref{lambdaL} we have the uniform upper bound: 
\begin{align}\label{part2}
\Lambda(p,q) (\rho +2) |p-q|^2  \lesssim
\begin{cases}
\frac{\sqrt{\pZ \qZ}}{|p-q|},   &\mbox{for} \, \rho \leq \frac{1}{8}, \\
\frac{\pZ}{\qZ} + \frac{\qZ}{\pZ},  	 &\mbox{for} \, \rho \geq \frac{1}{8}.
\end{cases}
\end{align}
\end{lemma}

\begin{proof}
We recall \eqref{lambdaL} and \eqref{lambdaS} to get
\begin{align*}
\left| \Phi^{ij}(p,q)  \right|
& \lesssim  \frac{(\rho +1)^2}{\pZ \qZ} \left(\rho (\rho +2)\right)^{-\frac{3}{2}}
	 \Big(\rho (\rho +2) + |p-q|^2   +  \rM \, \pZ \qZ\Big)
\\
& \lesssim  \frac{(\rho +1)^{1/2}}{\pZ \qZ}\rho^{-3/2}  \Big(\rho (\rho +2) + |p-q|^2   +  \rM \, \pZ \qZ\Big).
\end{align*}
Next, from \eqref{sestimate} we have $\rho \geq \frac{|p-q|^2}{2 \pZ \qZ}$. Therefore on $\rho < 1/8$ we have
\begin {align*}
\left| \Phi^{ij}(p,q)  \right|
& \lesssim
\frac{(\rho +1)^{1/2}}{\pZ \qZ}\left( \left( \frac{\pZ \qZ}{|p-q|^2} \right)^{3/2} |p-q|^2   +  \left( \frac{\pZ \qZ}{|p-q|^2} \right)^{1/2} \pZ \qZ\right)
\\
& \lesssim \frac{\sqrt{\pZ \qZ}}{|p-q|}.
\end{align*}
Further note that on $\rho \ge 1/8$ we have that
\begin{multline}\notag
\left| \Phi^{ij}(p,q)  \right| 
\lesssim
\frac{(\rho +1)^{1/2}}{\pZ \qZ}\rho^{-3/2}  \Big(\rho \pZ \qZ + |p-q|^2   +  \rM \, \pZ \qZ\Big)
\\
\lesssim
1+\frac{\pZ}{\qZ} +\frac{\qZ}{\pZ }
\lesssim
\frac{\pZ}{\qZ} +\frac{\qZ}{\pZ }.
\end{multline}
These two together prove \eqref{part1}.  Now recall \eqref{lambdaL}.
Next we prove \eqref{part2} for $\rho \leq \frac{1}{8}$:
\begin{align*}
\Lambda(p,q) (\rho +2) |p-q|^2  
& =  \frac{(\rho +1)^2}{\pZ \qZ} \left(\rho (\rho +2)\right)^{-\frac{3}{2}} (\rho +2) |p-q|^2
\\
& \lesssim   \frac{1}{\pZ \qZ} \left(\frac{\pZ \qZ}{|p-q|^2} \right)^{\frac{3}{2}}  |p-q|^2
\\
& \lesssim \frac{\sqrt{\pZ \qZ}}{|p-q|}. 
\end{align*}
On the other hand, if $\rho > \frac{1}{8}$, then
\begin{align*}
\Lambda(p,q) (\rho +2) |p-q|^2  
& = \left( \frac{\rho+1}{\rho+2} \right)^2  \left(\frac{\rho+2}{\rho} \right)^{\frac{3}{2}}  \frac{|p-q|^2}{\pZ \qZ}
\\
& \lesssim \frac{|p-q|^2}{\pZ \qZ} \lesssim \frac{(\pZ)^2 + (\qZ)^2}{\pZ \qZ},
\end{align*}
which establishes \eqref{part2}. 
\end{proof}

\subsection{Uniform bounds for $a^{ij}(h)$}\label{sec:uniformaijBOUNDS}
With the bounds on the kernel from the previous section, 
we can now establish the uniform upper bounds for $a^{ij}(h)$:

\begin{lemma}\label{lemma.upper.bd}  Let $\xi \in \threed$, and $h \in L^1_s(\threed)\cap L^3(\threed)$ with $s>2$.  Then
$$
a^{ij}(h) \xi_i \xi_j \le C_1 | \xi |^2.
$$  
Here $C_1>0$ is explicitly computable and $C_1 = C_1(\| h \|_{L^1_s(\threed)}, \| h \|_{L^3(\threed)})$.

Alternatively if we only have $g\in L^1_1(\threed)\cap L^3(\threed)$ then we have that
$$
a^{ij}(g) \xi_i \xi_j \le \pZ \tilde{C}_1 | \xi |^2.
$$  
Here $\tilde{C}_1>0$ is explicitly computable and $\tilde{C}_1 = \tilde{C}_1(\| g \|_{L^1_1(\threed)}, \| g \|_{L^3(\threed)})$.
\end{lemma}

We also have the uniform pointwise lower bound as follows:

\begin{lemma}\label{lemma.lower.bd}
For $h=h(p)\ge 0$ satisfying $\int_{\threed} h(p) dp > 0$, $M_1(h) \le \overline{M}$ and $\overline{H} (h) \le \overline{H}$ we have  the following estimate
$$
a^{ij}(h) \xi_i \xi_j \ge C_2 | \xi |^2.
$$ 
Here the constant $C_2>0$ is explicitly computable and only depends upon $\int_{\threed} h(p) dp$,  $\overline{M}>0$ and $\overline{H}>0$.
\end{lemma}

We will first prove Lemma \ref{lemma.lower.bd}, and afterwards we will prove Lemma \ref{lemma.upper.bd}.

\begin{proof}[Proof of Lemma \ref{lemma.lower.bd}]
The proof of the lower bound is an application of the proof of Theorem \ref{entropy.thm} that was recently completed above.  We will follow that proof very closely.  We note from \eqref{lambdaL} and \eqref{lambdaS} that 
\begin{equation}\notag
a^{ij}(h) \xi_i \xi_j 
=
\int_{\threed} dq ~ \Lambda(p,q) \left( S^{ij} \xi_i \xi_j \right) h(q).
\end{equation}
Here as in \eqref{phi.lower.ab} we have that 
\begin{equation}\notag
\Lambda(p,q) \left( S^{ij} \xi_i \xi_j \right)
\ge \frac{( \rM+1)^2}{\pZ \qZ}
\left(  \qS \rM \right)^{-1/2}
\frac{\left|\left( \qZ p - \pZ q \right)   \times \xi \right|^2}{\left| \qZ p - \pZ q \right|^2}.
\end{equation}
And then due to the identity \eqref{cross.prod.id} we have that 
\begin{equation}\label{aij.lower.proof}
a^{ij}(h) \xi_i \xi_j 
\ge 
\int_{\threed} dq ~  h(q) ~ \frac{( \rM+1)^2}{\pZ \qZ}
\left(  \qS \rM \right)^{-1/2}~ \left| \frac{p}{\pZ} - \frac{q}{\qZ} \right|^{-2}
~
\sum_{i,j=1}^3 \left| q_{ij}(p,q) \right|^2.
\end{equation}
Above the $q_{ij}(p,q)$ is not the same as in the proof of Theorem \ref{entropy.thm} even though we use the same notation.  
Here again recalling \eqref{cross.prod.id} then $q_{ij}(p,q)$ is defined as
\begin{align*}
 q_{ij}(p,q)  
&=
\left( \pii - \qi \right) \xi_j-
\left( \pj - \qj \right) \xi_i.
\end{align*}
From this point we will follow the proof of Theorem \ref{entropy.thm}, in an easier case.  In particular we define
$$
X_1  \eqdef  \pii \xi_j -  \pj \xi_i, ~
X_2   \eqdef \xi_i, ~
X_3  \eqdef \xi_j.
$$
By integrating against the three integrals in \eqref{three.integrals.tab} then we can establish a (simpler) linear system.    We use Cramer's formula as in just above \eqref{cramer.formula} to establish that
\begin{align*}
&\xi_i
= \Delta_\phi(h)^{-1} 
 \det  \left( \int_\threed \phi(\frac{|q|^2}{2}) h(q) \begin{bmatrix}
1& \qi & q_{ij}   \\
\qi & \left(\qi\right)^2 &q_{ij}   \frac{q_i}{\qZ}   \\
\qj & \qij & q_{ij}  \frac{q_j}{\qZ}  
\end{bmatrix}
\d q \right). 
\end{align*}
And then as in \eqref{cramer.formula} we have the estimate
\begin{equation} \notag
\left| \xi_i \right| 
\lesssim
  \Delta_\phi(h)^{-1} 
\left( \int_\threed \phi(\frac{|q|^2}{2}) h(q)  \d q \right)^2 
\left(
 \int_\threed \phi(\frac{|q|^2}{2}) h(q)
  |q_{ij}|
\d q
\right).
\end{equation}
We then square the above and multiply and divide by the square root of \eqref{A.def.thing} inside the integral containing $|q_{ij}|$, and use Cauchy-Schwartz to obtain that
\begin{multline*} \notag
\left| \xi_i \right|^2 
\lesssim
  \Delta_\phi(h)^{-2} 
\left( \int_\threed \phi(\frac{|q|^2}{2}) h(q)  \d q \right)^4 
\left(
 \int_\threed \phi(\frac{|q|^2}{2}) h(q)
  |q_{ij}|^2 A
\d q
\right)
\\
\times
\left(
 \int_\threed \phi(\frac{|q|^2}{2})^2 h(q)
 A^{-1}
\d q
\right).
\end{multline*}
However 
$
 \int_\threed \phi(\frac{|q|^2}{2})^2 h(q)
 A^{-1}
\d q
\lesssim 1
$
as in  the proof of \eqref{claim.one.proof}.  Then summing the above, the proof follows from \eqref{aij.lower.proof} and Lemma \ref{determinant}.
\end{proof}

Now we will prove Lemma \ref{lemma.upper.bd}.

\begin{proof}[Proof of Lemma \ref{lemma.upper.bd}]
Suppose without loss of generality that $| \xi | = 1$.  We have that
\begin{multline}\notag
a^{ij}(h) \xi_i \xi_j 
=
\int_{\threed} dq ~  \left( \Phi^{ij} \xi_i \xi_j \right) h(q)
\\
\le
\int_{\threed} dq ~  \Lambda(p,q)  \left| q^0p -  p^0 q \right|^2 
\frac{\left|\left( \qZ p - \pZ q \right)   \times \xi \right|^2}{\left| \qZ p - \pZ q \right|^2}  h(q)
\\
\le
\int_{\threed} dq ~  \Lambda(p,q)  \left| q^0p -  p^0 q \right|^2   h(q).
\end{multline}
In this calculation we used \eqref{eigenvalues.expansion}, and the fact that $\rho(\rho+2) \le \left| q^0p -  p^0 q \right|^2$.  
The reverse of this type of inequality was used in  \eqref{kernel.lower1}.  We will estimate this upper bound below.

In particular we first recall that
$$
\left| q^0p -  p^0 q \right|^2
=
\rho(\rho+2) + \left| p \times q \right|^2.
$$
Now we plug this into the above and estimate each of the terms on the right individually.  

In particular for the $\rho(\rho+2)$ term we have
\begin{multline}\notag
\int_{\threed} dq ~  \Lambda(p,q)  \rho(\rho+2)    h(q)
=
\int_{\threed} dq ~  \frac{( \rM+1)^2}{\pZ \qZ}
\left(  (\rho+2) \rM \right)^{-3/2}  \rho(\rho+2)    h(q)
\\
\le
\int_{\threed} dq ~  \frac{( \rM+1)^{3/2}}{\pZ \qZ}
 \rM^{-1/2}   h(q).
\end{multline}
We will estimate this upper bound on several different regions.  Firstly if $\rho \ge 1/8$ then we use $\rho \le \pZ \qZ$ to obtain
$$
\int dq ~  \frac{( \rM+1)^{3/2}}{\pZ \qZ}
 \rM^{-1/2}   h(q)
 \le 
 \int_{\threed} dq ~  h(q).
$$
This is the upper bound that we will use in this regime.  Next if $\rho < 1/8$ then we have using \eqref{sestimate} that
$$
\int dq ~  \frac{( \rM+1)^{3/2}}{\pZ \qZ}
 \rM^{-1/2}   h(q)
\lesssim
 \int_{\threed} dq ~\frac{1}{\pZ \qZ}~ \left(\frac{\sqrt{\pZ \qZ}}{|p-q|} \right)  h(q)
 \lesssim
 \int_{\threed} dq ~ \frac{h(q)}{|p-q|}.
$$
Now we further split into $|p-q| \ge 1$ and $|p-q|\le 1$.  On $|p-q| \ge 1$ we use Young's inequality as 
\begin{equation}\label{young.ineq.upper}
\left\| h * \frac{1}{|\cdot|} 1_{\{|\cdot|\leq 1\}} \right\|_{L^\infty} 
\leq \left\| h \right\|_{L^r(\threed)} 
\left\| \frac{1}{|\cdot|} 1_{\{|\cdot|\leq 1\}} \right\|_{L^{r'}(\threed)}
 \lesssim  
 \left\| h \right\|_{L^r(\threed)}.
\end{equation}
Here $1=\frac{1}{r}+\frac{1}{r'}$ and we require $r'<3$ or equivalently $r>3/2$.  We conclude that 
\begin{equation}\notag
\int_{\threed} dq ~  \Lambda(p,q)  \rho(\rho+2)    h(q)
\le
\| h \|_{L^1(\threed)} + \| h \|_{L^r(\threed)}.
\end{equation}
This concludes our estimates for the $\rho(\rho+2)$  terms.

For the $\left| p \times q \right|^2$ terms above we have
\begin{multline}\notag
\int_{\threed} dq ~  \Lambda(p,q)  \left| p \times q \right|^2   h(q)
=
\int_{\threed} dq ~  \frac{( \rM+1)^2}{\pZ \qZ}
\left(  (\rho+2) \rM \right)^{-3/2}  \left| p \times q \right|^2    h(q)
\\
\le
\int_{\threed} dq ~  \frac{( \rM+1)^{1/2}}{\pZ \qZ}
 \rM^{-3/2} \left| p \times q \right|^2   h(q).
\end{multline}
This is the general upper bound that we will use.  Now if $\rho>1/8$ then using \eqref{sestimate}  we have 
\begin{multline}\notag
\int dq ~  \frac{( \rM+1)^{1/2}}{\pZ \qZ}
 \rM^{-3/2} \left| p \times q \right|^2   h(q)
 \le
 \int_{\threed} dq ~  \frac{\left| p \times q \right|^2}{\pZ \qZ}
 \rM^{-1}  ~  h(q)
 \\
  \le
 \int_{\threed} dq ~  \frac{\left| p \times q \right|^2}{\pZ \qZ}
\left( \frac{\pZ \qZ}{|p\times q|^2} \right)  ~  h(q)
  \le
 \int_{\threed} dq   ~  h(q).
\end{multline}
Alternatively if $\rho < 1/8$ then $\pZ \le 2$ implies $\qZ \le 5$, and these conditions also imply $|p-q| \le 2$.  
This holds since the  estimate \eqref{sestimate}  and the Cauchy-Schwartz inequality imply 
\begin{align}\label{compare.pq.est}
\qZ \leq \sqrt{2\rho \pZ \qZ} + \pZ 
\leq  \sqrt{\frac{\qZ}{2}} + 2
\leq \frac{1}{2} \left( \qZ + \frac{1}{2}\right) +2,
\end{align}
and so $\qZ \leq \frac{9}{2} \leq 5$. 
In addition, together with  \eqref{sestimate}, this implies
\begin{align*}
\{ \rho \leq \frac{1}{8}\} 
\subset \{ |p-q|^2 \leq 2\rho \pZ \qZ \leq \frac{5}{2} \} 
\subset \{ |p-q| \leq 2 \}.
\end{align*}
Then, on this region, we further use the lower bound in \eqref{sestimate} to obtain
\begin{multline}\notag
\int dq ~  \frac{( \rM+1)^{1/2}}{\pZ \qZ}
 \rM^{-3/2} \left| p \times q \right|^2   h(q)
 \le
 \int_{\threed} dq ~  \frac{ \left| p \times q \right|^2 }{\pZ \qZ}
 \rM^{-3/2} ~  h(q)
 \\
  \le
 \int_{\threed} dq ~  \frac{\left| p \times q \right|^2}{\pZ \qZ}
\left( \frac{\pZ \qZ}{|p\times q|^2+|p-q|^2} \right)^{3/2}  ~  h(q)
  \le
 \int_{\threed} dq   ~  \frac{h(q)}{|p-q|}.
\end{multline}
Now as in \eqref{young.ineq.upper} we obtain that 
\begin{equation}\notag
 \int_{\threed} dq   ~  \frac{h(q)}{|p-q|}
 \le 
\| h \|_{L^1(\threed)} + \| h \|_{L^r(\threed)}.
\end{equation}
This holds for any $r>3/2$.  This is the main estimate in this region.

Lastly if $\rho < 1/8$ and $\pZ\ge 2$ then $\qZ \ge 1$, and $\pZ \approx \qZ$.   Namely, 
using again  \eqref{sestimate}  and the Cauchy-Schwartz inequality one can easily see that
$$
\qZ \geq \pZ -|p-q| \geq \pZ - \sqrt{2\rho \pZ \qZ} \geq \pZ - \sqrt{\frac{1}{4} \pZ \qZ} \geq \pZ - \frac{1}{4}(\pZ + \qZ).
$$
Therefore, $\qZ \geq \frac{3}{5} \pZ \geq \frac{6}{5}\geq 1$. 
In addition, one can similarly show that $\pZ \geq \frac{3}{5} \qZ$, which implies
\begin{align}\label{D2-comp}
\frac{5}{3} \pZ \geq	 \qZ \geq \frac{3}{5} \pZ.
\end{align}
In this regime, using \eqref{sestimate}, we have
\begin{multline}\notag
\int dq ~  \frac{( \rM+1)^{1/2}}{\pZ \qZ}
 \rM^{-3/2} \left| p \times q \right|^2   h(q)
 \\
  \le
 \int_{\threed} dq ~  \frac{\left| p \times q \right|^2}{\pZ \qZ}
\left( \frac{\pZ \qZ}{|p\times q|^2+|p-q|^2} \right)^{3/2}  ~  h(q)
  \le
 \int_{\threed} dq   ~  \frac{\sqrt{\pZ\qZ}}{|p-q|} h(q).
\end{multline}
On the one hand we can estimate this, using $\pZ \approx \qZ$, as above as 
$$
 \int dq   ~  \frac{\sqrt{\pZ\qZ}}{|p-q|} h(q)
 \lesssim
 \pZ
 \left( \| h \|_{L^1(\threed)} + \| h \|_{L^r(\threed)} \right).
$$
This would give the second estimate in Lemma \ref{lemma.upper.bd}.

On the other hand, splitting into $|p-q| \ge 1$ and $|p-q| \le 1$, and using $\pZ \approx \qZ$, as in \eqref{young.ineq.upper} we obtain
$$
 \int dq   ~  \frac{\sqrt{\pZ\qZ}}{|p-q|} h(q)
 \lesssim
 \left( \| h \|_{L^1_1(\threed)} + \| h \|_{L^r_1(\threed)} \right).
$$
This above holds for $r>3/2$ as in the previous case.  Further by interpolation inequality, see for example \cite[Proposition 6]{MR3369941},  we can bound 
$$
\| h \|_{L^r_1(\threed)} \lesssim \| h \|_{L^1_s(\threed)}^\beta \| h \|_{L^3(\threed)}^{1-\beta}.
$$
Here we take $r=\frac{3}{2} + \epsilon$ where
\begin{align*}
& 1 = \beta s + (1-\beta)0,\\
& \frac{1}{r} = \frac{\beta}{1} + \frac{1-\beta}{3}.
\end{align*}
So that we need to use the weight $s=1/\beta$ where $\beta =\frac{3-2\epsilon}{2(3+2\epsilon)}$.  Since $\beta < \frac{1}{2}$ then we need to use $s> 2$.  Note that this interpolation can be proven directly from the H{\"o}lder inequality.  Collecting all of these estimates completes the proof.
\end{proof}

\section{Propagation of high moment bounds}\label{sec.prop.high.moment}

The main result in this section will be to prove Theorem \ref{prop-mom}. Before we proceed, recall from \eqref{momentTnotation} moment notation
\begin{align*}
M_k(f,T) \eqdef  \esssup_{t\in[0,T]} \int_\threed f(t,p) (1+|p|^2)^k dp.
\end{align*}
   Theorem \ref{prop-mom} will be proved by inductively applying the following lemma:
\begin{lemma}\label{lem-mom}
Let $k>1$. Suppose that $M_k(f,0)<\infty$  
 and   $M_{k-\frac{1}{2}}(f,T)<\infty$.  Then
$$ M_k(f,T) <C, $$
where $C<\infty$ depends only on $T,k$, the collision kernel $\Phi$, $\Q_T(f)  \eqdef  \int_0^T \|f\|_{L^3(\threed)} dt$, the initial moment $M_k(f,0)$ and the moments $M_{\frac{8}{11}(k-1)}(f,T)$ and $M_{k-\frac{1}{2}}(f,T)$.
\end{lemma}

\begin{proof}[Proof of Lemma \ref{lem-mom}]
Let  $\alpha \in C^\infty_c(\threed)$, $0\le \alpha \le 1$, be such that $\alpha|_{[0,1]} = 1$ and $\alpha|_{[0,2]^c}= 0$ and let $\eta \in (0,1)$. Define 
\begin{align*}
\varphi(p) \eqdef (1+|p|^2)^k \alpha(\eta\sqrt{1+|p|^2}).
\end{align*}
We let $\partial_i = \partial_{p_i}$.  Then for any $i, j \in \{ 1,2,3\}$ we have
\begin{align*}
\partial_i \varphi(p) & = 2 k p_i (1+|p|^2)^{k-1} \alpha(\eta\sqrt{1+|p|^2}) 
	+ \eta p_i (1+|p|^2)^{k-\frac{1}{2}} \alpha'(\eta\sqrt{1+|p|^2}),
\end{align*}
and letting $\partial_{ij} = \partial_{p_i}\partial_{p_j}$ we have
\begin{align*}
\partial_{ij} \varphi(p) & = 2k\Big( \delta_{ij} + 2(k-1) p_i p_j  (1+|p|^2)^{-1}\Big) (1+|p|^2)^{k-1} \alpha(\eta\sqrt{1+|p|^2})\\
& \quad + \Big( \delta_{ij} \eta \pZ + (4k-1)\eta p_i p_j  (1+|p|^2)^{-\frac{1}{2}} \Big)  (1+|p|^2)^{k-1}  \alpha'(\eta \pZ)\\
& \quad + \eta^2 p_i p_j (1+|p|^2)^{k-1} \alpha''(\eta\sqrt{1+|p|^2}).
\end{align*}
We denote the norm $\| \cdot \|_{L^\infty}$ by $\| \cdot \|_{\infty}$.
Then, for some constant $C(k)$ that depends on $k, \|\varphi'\|_\infty$ and $\| \varphi'' \|_\infty$, we have
\begin{align*}
|\partial_{p_i p_j} \varphi(p)| &\leq C(k)  (1+|p|^2)^{k-1}.
\end{align*}

We use the weak formulation of the collision operator in \eqref{weak.form.landau} to obtain
\begin{multline}
\int_\threed  f(T,p)  \varphi(p) dp  - \int_\threed f(0,p) \varphi(p) dp
= \int_0^T dt   \int_\threed \mathcal{C}(f,f)(p) \, \varphi(p) \, dp 
\\
 = 
\frac{1}{2} \int_0^T dt   \int_\threed   \int_\threed f(p) f(q) \Phi^{ij}(p,q) 
	\left( \partial_{p_j p_i}\varphi(p) +  \partial_{q_j q_i}\varphi(q)\right) dqdp 
\\
\qquad 
+
\int_0^T dt   \int_\threed  \int_\threed f(p) f(q)   \Lambda(p,q)  (\rM+2) \left( q_i- p_i\right) 
	\left( \partial_{p_i}\varphi(p)  - \partial_{q_i}\varphi(q)\right) dqdp
\\
 \lesssim 
\int_0^T dt\int_\threed   \int_\threed f(p) f(q) \left|  \Phi^{ij}(p,q) \right| (1+|p|^2 +|q|^2)^{k-1} dq dp
\\
\qquad 
+
\int_0^T dt  \int_\threed  \int_\threed f(p) f(q)   \Lambda(p,q)  (\rM+2) |p-q|^2 (1+|p|^2 +|q|^2)^{k-1} dq dp.
\label{pf-weak}
\end{multline}

Estimating \eqref{pf-weak} using \eqref{part1} and \eqref{part2} we have
\begin{align}
\int_\threed & f(T,p)  \varphi(p) dp  - \int_\threed f(0,p) \varphi(p) dp 
\nonumber \\
& \lesssim  \int_0^T\iint_{\{\rho \leq \frac{1}{8} \}} f(p) f(q) \frac{\sqrt{\pZ \qZ}}{|p-q|} (1+|p|^2 +|q|^2)^{k-1} dq dp
\nonumber \\
&\quad +
 \int_0^T\iint_{\{\rho > \frac{1}{8} \}} f(p) f(q)  \left( \frac{\pZ}{\qZ} + \frac{\qZ}{\pZ}\right) (1+|p|^2 +|q|^2)^{k-1} dq dp
 \nonumber\\
& \lesssim  \int_0^T\iint_{\{\rho \leq \frac{1}{8}, \pZ \leq 2\}} f(p) f(q) \frac{\sqrt{\pZ \qZ}}{|p-q|} (1+|p|^2 +|q|^2)^{k-1} dq dp
  \label{pf-weak-a} \\
& \quad + \int_0^T\iint_{\{\rho \leq \frac{1}{8}, \pZ \geq 2\}} f(p) f(q) \frac{\sqrt{\pZ \qZ}}{|p-q|} (1+|p|^2 +|q|^2)^{k-1} dq dp
\nonumber \\
&\quad +
 \int_0^T\iint_{\{\rho > \frac{1}{8} \}} f(p) f(q)  \left( \frac{\pZ}{\qZ} + \frac{\qZ}{\pZ}\right) (1+|p|^2 +|q|^2)^{k-1} dq dp  
 \nonumber\\
& =: I_1 + I_2 + I_3, \nonumber
\end{align}
where integrals $I_1, I_2$ and $I_3$ correspond to the domains $D_1=\{\rho \leq \frac{1}{8}, \pZ \leq 2 \}$,
 $D_2=\{\rho \leq \frac{1}{8}, \pZ \geq 2 \}$ and  $D_3=\{\rho \geq \frac{1}{8} \}$, respectively. We will estimate them separately.

First note that the set $D_1$ is a subset of $\{ \rho \leq \frac{1}{8}, \pZ \leq 2, \qZ\leq 5, |p-q| \leq 2 \}$
as in \eqref{compare.pq.est}.  Therefore for some constant $C(k)$ depending only on $k$ we have
\begin{align*} 
& I_1 \eqdef  \int_0^T \iint_{D_1} f(p) f(q) \frac{\sqrt{\pZ \qZ}}{|p-q|} \left( 1+|p|^2 +|q|^2 \right)^{k-1} dq dp dt \\  
& \qquad \leq C(k) \int_0^T \int_\threed f(p) 
	\left( \int_\threed f(q) \frac{1}{|p-q|} 1_{\{|p-q|\leq 2\}} \, dq \right)dp dt \\ 
& \qquad \leq C(k) \int_0^T \|f\|_{L^1(\threed)} 
	\left\| f( \cdot) * \frac{1}{|\cdot|} 1_{\{|\cdot|\leq 2\}} \right\|_{L^\infty(\threed)} dt.
\end{align*}
By Young's inequality for convolutions
$$ 
\left\| f( \cdot) * \frac{1}{|\cdot|} 1_{\{|\cdot|\leq 2\}} \right\|_{L^\infty(\threed)} 
\leq \left\| f \right\|_{L^3(\threed)} 
\left\| \frac{1}{|\cdot|} 1_{\{|\cdot|\leq 2\}} \right\|_{L^{\frac{3}{2}}(\threed)}.
$$
Since the second term on the right-hand side is a finite number,  we can further estimate the integral in the domain $D_1$ as follows
\begin{align}
I_1 &\lesssim  C(k) \|f\|_{L^\infty_T L^1(\threed)} \int_0^T \|f(t,\cdot) \|_{L^3(\threed)} \, dt \label{I1} \\
& = C(k) \|f_0\|_{ L^1(\threed)} \Q_{T}(f) \nonumber \\
& = C\left(k,f_0, \Q_T(f) \right) < \infty. \nonumber
\end{align}

The domain $D_2$ is a subset of $\{(p,q)\in\R^{3+3}: \rho \leq \frac{1}{8}, ~\pZ \geq 2, ~ \qZ\geq 1, \pZ \approx \qZ \}$ as in \eqref{D2-comp}.  The comparability \eqref{D2-comp} further implies 
\begin{align*}
(\pZ)^2 \leq 1 +|p|^2 + |q|^2 \leq \frac{8}{3} (\pZ)^2\\
(\qZ)^2 \leq 1 +|p|^2 + |q|^2 \leq \frac{8}{3} (\qZ)^2,
\end{align*}
so the weight function $(1 +|p|^2 + |q|^2)^{k-1}$ can be estimated as follows for any $l\in\R$:
\begin{align*}
(1 +|p|^2 + |q|^2)^{k-1} 
& = (1 +|p|^2 + |q|^2)^{l/2} (1 +|p|^2 + |q|^2)^{k-1-l/2}\\
& \leq \sup(1, (8/3)^{l/2}) \sup(1, (8/3)^{k-1-l/2}) (\pZ)^l (\qZ)^{2k-2-l}.
\end{align*}
Therefore,  the integral in the domain $D_2$ can be bounded as follows
\begin{multline}\label{I2-a}
I_2 \eqdef 
\int_0^T \iint_{D_2} f(p) f(q) \frac{\sqrt{\pZ \qZ}}{|p-q|} \left( 1+|p|^2 +|q|^2 \right)^{k-1} dq dp dt   
\\
  \leq C(k,l) \int_0^T \iint_{D_2} f(p) f(q) \frac{1}{|p-q|} (\pZ)^{2(\frac{l}{2}+\frac{1}{2})} (\qZ)^{2(k-1-\frac{l}{2}+\frac{1}{2})} dp dq dt 
\\
\leq 
C(k,l) \int_0^T 
	\| \langle \cdot \rangle^{2(\frac{l}{2}+\frac{1}{2})} f\|_{L^{r'}(\threed)}
	\| \langle \cdot \rangle^{2(k-1-\frac{l}{2}+\frac{1}{2})} f\|_{L^1(\threed)}
	\left\| \frac{1}{|\cdot|} 1_{\{|\cdot|\leq 2\}} \right\|_{L^r(\threed)} dt, 
\end{multline}
for a pair of H\"older conjugate indexes $r$ and $r'$ and for some constant $C(k,l)$ that depends on $k$ and $l$. In the last inequality we used H\"older and Young inequalities. We require that $r<3$ in order for the last term to be finite.  Note that here we also only considered the case when $|p-q| \le 2$.  However the case $|p-q| \ge 2$ satisfies a better estimate involving only the weighted $L^1$ norms, as in the $I_3$ term in \eqref{I3} below.

Now, by interpolation inequality, see for example \cite[Proposition 6]{MR3369941}, we can bound the weighted $L^{r'}$ norm appearing above in terms of the weighted $L^1$ and the $L^3$ norm as follows
\begin{align*}
\| \langle \cdot \rangle^{2(\frac{l}{2}+\frac{1}{2})} f\|_{L^{r'}(\threed)}
	\leq \| f\|_{L^3(\threed)}^\beta  \,
		\| \langle \cdot \rangle^{2(k-1-\frac{l}{2}+\frac{1}{2})} f\|_{L^1(\threed)}^{1-\beta},
\end{align*}
where
\begin{align*}
& \frac{l}{2}+\frac{1}{2} = (1-\beta) \left(k-1-\frac{l}{2}+\frac{1}{2}\right),\\
& \frac{1}{r'} = \frac{\beta}{3} + 1-\beta.
\end{align*}
We note here also that this interpolation can be proven directly from the standard H{\"o}lder inequality.

For $\left\| \frac{1}{|\cdot|} 1_{\{|\cdot|\leq 2\}} \right\|_{L^r(\threed)}$ to be finite we need $r<3$, so let $r = 12/5$. Then $r' = 12/7$, $\beta = \frac{5}{8}$ and $l = \frac{(2k-1)(1-\beta) -1}{2-\beta}$, so in particular we have
\begin{align}\label{l}
k-1-\frac{l}{2} + \frac{1}{2} = \frac{8}{11}(k-1).
\end{align}
Therefore, continuing from \eqref{I2-a} we have
\begin{align} \label{I2}
I_2 & \lesssim C(k,l) \int_0^T  \| f\|_{L^3(\threed)}^\beta  \,
		\| \langle \cdot \rangle^{2(k-1-\frac{l}{2}+\frac{1}{2})} f\|_{L^1(\threed)}^{2-\beta} dt 
 \\
& \leq C(k,l)  M^{2-\beta}_{\frac{8}{11}(k-1)}(f,T) \int_0^T  \| f\|_{L^3(\threed)}^{\frac{1+\e}{2}} dt 
\nonumber \\
& \leq  C\left(k,l, M_{\frac{8}{11}(k-1)}(f,T)\right) \left(\int_0^T  \| f\|_{L^3(\threed)} dt\right)^{\frac{1+\e}{2}}
\left( \int_0^T  dt \right)^{\frac{1-\e}{2}}
\nonumber \\
& \leq  C\left(k,l, T, \e, M_{\frac{8}{11}(k-1)}(f,T), \Q_T(f) \right) < \infty. \nonumber
\end{align}

\bigskip
Finally, in the domain $D_3 = \{(p,q)\in\R^{3+3}: \rho \geq \frac{1}{8} \}$ we have
\begin{align} \label{I3}
I_3 = & \int_0^T \iint_{D_3} f(p) f(q) \left( \frac{\pZ}{\qZ} + \frac{\qZ}{\pZ} \right) \left( 1+|p|^2 +|q|^2 \right)^{k-1} dq dp dt \\  \nonumber
& \leq 2  \int_0^T \iint_{D_3} f(p) f(q)
	\left( (\pZ)^{2(k-1+\frac{1}{2})} (\qZ)^{2(k-1+\frac{1}{2})} \right) dq dp dt \\ \nonumber
& \leq 2 T \left( M_{k-\frac{1}{2}}{(f,T)}\right)^2. \label{eq-D3}
\end{align}

By gathering estimates \eqref{pf-weak-a}, \eqref{I1}, \eqref{I2} and \eqref{I3} we get
\begin{align*}
\int_\threed f(T,p) \varphi(p) dp \leq C\left(k,l, T,  M_{\frac{8}{11}(k-1)}(f,T), \Q_T(f),  M_{k-\frac{1}{2}}{(f,T)} \right)<\infty.
\end{align*}
Finally let $\eta \rightarrow 0$ to conclude the proof of the lemma.
\end{proof}

With Lemma \ref{lem-mom} in hand, we now proceed to the proof of Theorem \ref{prop-mom}.

\begin{proof}[Proof of Theorem \ref{prop-mom}]
Let $\e \in (0,1)$. Define the sequence
\begin{align*}
&k_0  =k, \\
&k_{n+1}  = k_n-\frac{1}{2}, \qquad \mbox{for} \, n\in \mathbb{N}.
\end{align*}
Note that the sequence is decreasing and we can find $n_0 \in \mathbb{N}$ such that $k_{n_0} >1$ and $k_{n_0 +1} \leq 1$. Then, starting from $k_{n_0}$ and using Lemma \ref{lem-mom} repeatedly, we get that $M_{k_n}(f,T) < \infty$ for all $n = n_0, n_0-1, ..., 0$, so that in particular $M_k(f,T) <\infty$.
\end{proof}

\section{Global existence of a true weak solution}\label{sec:TWS}

In this section we use the approach from \cite{MR1650006} and \cite{MR3369941}.  We recall also \cite{MR2068110}.  We will give a sketch of the standard construction of global in time weak solutions to the relativistic Landau equation using our estimates from the previous sections.  

\subsection{Existence for a regularized problem}\label{sec.regularized.sec}
  We recall the kernel \eqref{kernelnormalized}: 
$$
\Phi^{ij}(p,q) = \Lambda(p,q)S^{ij}(p,q),
$$
with \eqref{lambdaL} and \eqref{lambdaS}.    We will smoothly approximate $\Lambda(p,q) \qS \rM $ by  $\Lambda_n$ such that $\Lambda_n \to \Lambda(p,q) \qS \rM $ pointwise as $n\to \infty$.  We remove the singularity of the kernel at $p=q$.  In particular we can choose
$$
\Lambda_n(p,q) \eqdef \frac{( \rM+1)^{2}}{\pZ \qZ}  \left(  \qS \rM +n^{-2}\right)^{-1/2}.
$$
Now let $n=1/\epsilon$ and define 
 $\Phi_{\epsilon}^{ij} (p,q) \eqdef \Lambda_{1/\epsilon}(p,q)  S^{ij}(p,q) / \qS \rM $.  We choose this decomposition so that $\Lambda_{1/\epsilon}(p,q)$ develops a first order singularity as $\epsilon \to 0$.  Further $S^{ij}(p,q) / \qS \rM$ is bounded due to  \eqref{lambdaS}.   Then $\Phi_{\epsilon} (p,q)$ satisfies the null space \eqref{nullphi} and the non-negativity \eqref{PositivePhi} with the same proof as in Lemma \ref{sij.ident.lem}.  Further $\Phi_{\epsilon} (p,q) \to \Phi (p,q)$ as ${\epsilon}\to 0$  on compact sets when $p\ne q$.  In particular we have that
\begin{equation}\label{r.convergence.est}
\int_{B(0,R)} dp ~ \int_{B(0,R)} dq ~ \left|\Phi^{ij}_{\epsilon} (p,q)  - \Phi^{ij} (p,q)   \right|^r \to 0, \quad \epsilon \to 0. 
\end{equation}
This convergence \eqref{r.convergence.est} holds for any $1 \le r < 3$ by the dominated convergence theorem.  
Also $\partial_{p_i} \Phi_{\epsilon} (p,q) \to \partial_{p_i}\Phi (p,q)$ and $\partial_{q_i} \Phi_{\epsilon} (p,q) \to \partial_{q_i}\Phi (p,q)$ pointwise as ${\epsilon}\to 0$  on compact sets when $p\ne q$.

Let $\mathcal{C}_{\epsilon}(h,g)(p)$ be the relativistic Landau operator \eqref{landauC} with kernel $\Phi_\epsilon(p,q)$ instead of $\Phi(p,q)$.  Analogous to $\mathcal{C}(h,g)(p)$ defined in \eqref{landauC}, we have
\begin{gather}\notag
\mathcal{C}_\epsilon(f,g)(p)
\eqdef
\partial_{p_i} \int_{\mathbb{R}^3}\Phi_{\epsilon}^{ij}(p,q)\left\{f(q) \partial_{p_j}  g(p)  -\partial_{q_j} f(q) g(p) \right\}
dq.
\end{gather}
Then, similar to \eqref{landauConservativeForm}, we introduce the following approximate problem:
\begin{equation}\label{truncatedCauchy.n}
\partial_t f^{\epsilon}
=
\partial_{p_i}\left( \TA_{\epsilon}^{ij}(f^{\epsilon}) \partial_{p_j}  f^{\epsilon}
+
\TB^{i}_{\epsilon}(f^{\epsilon})
~ f^{\epsilon}
\right)
+
{\epsilon}  \Delta f^{\epsilon},
\end{equation}
where the coefficients are
\begin{equation}\label{aijNdefN}
\TA^{ij}_{\epsilon}(f^{\epsilon})
\eqdef
\int_{\mathbb{R}^3}\Phi_{\epsilon}^{ij}(p,q) f^{n}(q)dq,
\end{equation}
and
\begin{equation}\notag
\TB^{i}_{\epsilon}(f^{\epsilon})
\eqdef
\int_{\mathbb{R}^3}\partial_{q_j} \Phi_{{\epsilon}}^{ij}(p,q) f^{{\epsilon}}(q)dq.
\end{equation}
%
%
%
%
%
This type of reduced parabolic system \eqref{truncatedCauchy.n} is well known to have global in time unique smooth solutions using the Schauder fixed point theorem.   We only give a very brief outline.  Essentially identical arguments are shown in detail in \cite{MR1737547} and \cite{MR2068110}.
We consider smooth initial data $f^{\epsilon}(0,p)=f_0^{\epsilon}(p)$ which satisfies
$$
f_0^{\epsilon}(p) 
\ge 
\alpha_1(\epsilon) e^{-\beta_1(\epsilon) \pZ}.
$$
For suitably chosen $\alpha_1, \beta_1 > 0$.    Then by the comparison principle, for a $D>0$, it can be shown that
\begin{equation}  \label{iterate.lower}
f^{\epsilon}(t,p) 
\ge \tilde{\alpha}_1(\epsilon) e^{-\tilde{\beta}_1(\epsilon) \pZ} e^{-D t}.
\end{equation}
And further
\begin{equation}\notag
\| f^{\epsilon}\|_{L^\infty([0,T]; L^1(\mathbb{R}^3_p)\cap W^{2,\infty} (\mathbb{R}^3_p))}
+
\| f^{\epsilon}\|_{W^{1,\infty}([0,T]; W^{-2,1} (\mathbb{R}^3_p))}
\le C(\epsilon).
\end{equation}
Here  $W^{k,p}$ are the standard Sobolev spaces.  Note that we assume the initial data $f_0^{\epsilon}(p)$ satisfies a high moment bound, and then this moment bound can be propagated in time as in Theorem \ref{prop-mom}, proven in Section \ref{sec.prop.high.moment}.  The solution to \eqref{truncatedCauchy.n} will also satisfy the high moment bound.  Then using the estimates in Section \ref{unif.est.sec.new}, also as in Section \ref{kernel.est.section}, we can further show that \eqref{aijNdefN} satisfies
\begin{equation}\notag
|\xi |^2 c_\epsilon \le \left(\TA^{ij}_{\epsilon}(f^{\epsilon})  +\epsilon \delta_{ij} \right) \xi^i \xi^j \le C_\epsilon |\xi |^2.
\end{equation}
Note that the lower bound in \eqref{iterate.lower} can also give another proof of the lower bound above using the eigenvalue expansion as in \cite{MR1773932}.  
For further details, one can see a very similar problem carefully described in the arguments from \cite[Section 5]{MR1737547}.

\subsection{Uniform estimates}\label{unif.est.sec.new}

We can readily observe that, for solutions to \eqref{truncatedCauchy.n}, we also have  a uniform conservation of the mass as
\begin{equation}\label{mass.conserve.epsilon}
\int_{\threed} f^{\epsilon}(t,p) \ dp  = \int_{\threed} f^\epsilon_0(p)  dp.
\end{equation}
This grants the uniform estimate $f^{\epsilon} \in L^\infty([0,T]; L^1(\mathbb{R}^3))$.
In addition following the calculation from \eqref{zero.calc.integration} for \eqref{truncatedCauchy.n} then the energy satisfies 
\begin{multline*}
\int_{\threed} f^{\epsilon}(t,p) \pZ dp  = \int_{\threed} f^{\epsilon}_0(p) \pZ dp + t {\epsilon} \int_{\threed} f^{\epsilon}(t,p) dp
\\
= \int_{\threed} f^{\epsilon}_0(p) \pZ dp + t {\epsilon} \int_{\threed} f^{\epsilon}_0(p) dp.
\end{multline*}
This grants the uniform estimate $f^{\epsilon} \in L^1([0,T]; L^1_1(\mathbb{R}^3))$.  
Additionally if we let $D_{\epsilon}(f^{\epsilon})$ be the entropy dissipation defined in \eqref{entropy.dissipation} with $\Phi^{ij}$ replaced by $\Phi^{ij}_{\epsilon}$ then the calculations as in \eqref{sec.rel.landau.htheorem} rigorously hold for solutions to \eqref{truncatedCauchy.n}.  We have that
\begin{equation}\notag
 H(f^{\epsilon}(t)) + \int_0^t D_{\epsilon}(f^{\epsilon}(s)) ds
 +
  2\epsilon \int_0^t ds \int_{\threed} dp ~ | \nabla \sqrt{f^{\epsilon}(t,p)}|^2
 \le  
 H(f^{\epsilon}_0).
\end{equation}
These estimates give uniform bounds on $f^{\epsilon} (t,p)$ in $L^1_1$ and $L\log L$.  In particular it is well known that 
$  \overline{H}(f^{\epsilon}(t)) \lesssim H(f^{\epsilon}(t))+1$.

Notice that Theorem \ref{entropy.thm} still holds for $D_{\epsilon}$ with with $\Phi^{ij}$ replaced by $\Phi^{ij}_{\epsilon}$.  Thus from the lower bound on $D_{\epsilon}(f^{\epsilon}(s))$ in Theorem \ref{entropy.thm} we obtain a uniform bound on
\begin{equation}\notag
\int_0^t ds \int_{\mathbb{R}^3} |\nabla \sqrt{f^{\epsilon}(s,p)}|^2 \, dp  \lesssim 1.
\end{equation}
This grants a uniform estimate $\sqrt{f^{\epsilon}} \in L^2([0,T];\dot{H}^1(\threed))$. 
Then we further recall the Sobolev inequality:
$$ 
\left(\int_{\mathbb{R}^3} |f^{\epsilon}(p)|^3 \, dp \right)^{1/3}
\lesssim
\int_{\mathbb{R}^3} |\nabla \sqrt{f^{\epsilon}(p)}|^2 \, dp, 
$$
which grants us the following uniform bound  $f^{\epsilon} \in L^1([0,T];L^3(\threed))$.

We consider the standard $L^2$ based Sobolev space $H^m_k(\mathbb{R}^3)$ to have $m$ derivatives and the $k$-th order polynomial momentum weight $\langle p \rangle^k$.  Then given $\varphi \in H^m_k(\mathbb{R}^3)$ we observe that
\begin{multline}\label{weak.dt.fepsilon}
\int_{\mathbb{R}^3} dp ~ \partial_t f^{\epsilon} \varphi
\\
=
\frac{1}{2}\int_{\mathbb{R}^3}
 \int_{\mathbb{R}^3} f^{\epsilon}(q) f^{\epsilon}(p) \Phi^{ij}_{\epsilon}(p,q)
\left( \partial_{p_j}\partial_{p_i}\varphi(p) +  \partial_{q_j}\partial_{q_i}\varphi(q)\right) dqdp
\\
+\int_{\mathbb{R}^3}
 \int_{\mathbb{R}^3} f^{\epsilon}(p) f^{\epsilon}(q)
 \left(\partial_{p_j}\Phi^{ij}_{\epsilon}(p,q) - \partial_{q_j}\Phi^{ij}_{\epsilon}(p,q)  \right)
\left( \partial_{p_i}\varphi(p)  - \partial_{q_i}\varphi(q)\right) dqdp
\\
+
\epsilon 
 \int_{\mathbb{R}^3} dp ~  f^{\epsilon}(p) ~ \partial_{p_i}\partial_{p_i}\varphi(p).
\end{multline}
Here we remark that it can be seen directly from the proof that the derivative computation \eqref{intermediate2} still holds for $\partial_{q_j}\Phi^{ij}_{\epsilon}(p,q)$ and   $\partial_{p_j}\Phi^{ij}_{\epsilon}(p,q)$ with $\Lambda$ replaced by $\Lambda_{1/\epsilon}(p,q) / \qS \rM $.  Therefore we obtain the bound  
\begin{multline}\notag
\left| \int_{\mathbb{R}^3} dp ~ \partial_t f^{\epsilon} \varphi  \right|
\lesssim  
\| \varphi\|_{W^{2,\infty}}
\int_{\mathbb{R}^3}
 \int_{B(0,R)} 
 f^{\epsilon}(q) f^{\epsilon}(p) 
\left|  \Phi^{ij}(p,q) \right|
 dqdp
\\
+\| \varphi\|_{W^{2,\infty}}\int_{\mathbb{R}^3}
 \int_{B(0,R)} 
 f^{\epsilon}(p) f^{\epsilon}(q)
\Lambda(p,q) (\rho +2) |p-q|^2
 dqdp
\\
+
\epsilon 
 \| \varphi\|_{W^{2,\infty}}  \int_{\mathbb{R}^3} dp ~  f^{\epsilon}_0(p) .
\end{multline}
Then, using Lemma \ref{upper.bd.lemma} with $r'=3$, then after integration in time the upper bounds are uniformly bounded.  We also use the continuous embedding of $W^{2,\infty}$ into $H^m$.  So that we can also conclude that $(\partial_t f^\epsilon)_{\epsilon >0}$ is uniformly bounded in 
$L^1([0,T]; (H^m_k(\mathbb{R}^3))')$ for suitable $m$ and $k$.

We use the notation $f\in \sqrt{\dot{H}^1(\mathbb{R}^3)}$ to mean that $\sqrt{f} \in \dot{H}^1(\mathbb{R}^3)$.  Then we observe that
$$
\sqrt{\dot{H}^1(\mathbb{R}^3)} \cap L^1_1(\mathbb{R}^3)
\subset 
L^1(\mathbb{R}^3) 
\subset 
(H^m_k(\mathbb{R}^3))'.
$$
The embedding of $\sqrt{\dot{H}^1(\mathbb{R}^3)} \cap L^1_1(\mathbb{R}^3) \subset L^1(\mathbb{R}^3)$ is compact, and the embedding 
$L^1(\mathbb{R}^3) 
\subset 
(H^m_k(\mathbb{R}^3))'$
is continuous.  
Next we use the compactness result \cite[Corollary 4]{MR916688} to observe that $( f^\epsilon)_{\epsilon >0}$ is relatively compact in $L^2([0,T]; L^1(\mathbb{R}^3) )$.   Therefore there exists a function $f \in L^2([0,T]; L^1(\mathbb{R}^3))$ and a subsequence of  $( f^\epsilon)_{\epsilon >0}$ such that 
$(f^\epsilon)_{\epsilon >0}$ converges to $f$ in $L^2([0,T]; L^1(\mathbb{R}^3) )$ and a.e. on $[0,T] \times \threed$.

Now we also have the following lemma:  

\begin{lemma} \label{upper.bd.lemma} 
We have the uniform inequality:
\begin{multline*}
\int_{\mathbb{R}^3}dq  \int_{B(0,R)}dp ~  g(p) h(q) ~ \left|\Phi^{ij}(p,q) \right|
\\
\lesssim
\| g \|_{L^1_1(B(0,R))} \| h \|_{L^1_{-1}(\threed)} 
+
\| g \|_{L^1_{-1}(B(0,R))} \| h \|_{L^1_1(\threed)}
\\
+
\min\{
\| h \|_{L^1(\threed)}
\| g \|_{L^1_1(B(0,R))},
\| h \|_{L^1_1(\threed)}
\| g \|_{L^1(B(0,R))}
\}
\\
+
\min\{
\| h \|_{L^{1}_1(\threed)}
\| g \|_{L^{r'}(B(0,R))},
\| h \|_{L^{r'}(\threed)}
\| g \|_{L^{1}_1(B(0,R))}
\}.
\end{multline*}
where we can choose any $r' \in (3/2, \infty]$.  Here the implicit constant can be chosen to be independent of $R>0$.  

Further, we can use \eqref{part2} in place of \eqref{part1} in the proof.  
Then this lemma also holds when $\left|\Phi^{ij}(p,q) \right|$ is replaced by $\Lambda(p,q) (\rho +2) |p-q|^2$.
\end{lemma}

\begin{proof}[Proof of Lemma \ref{upper.bd.lemma}]
We will use the bounds in \eqref{part1}.  Then we have
\begin{multline*}
\int_{\mathbb{R}^3}dq  \int_{B(0,R)}dp~  g(p) h(q) \left|\Phi^{ij}(p,q) \right|
\lesssim
\int_{\mathbb{R}^3}dq  \int_{B(0,R), \rho < \frac{1}{8}}dp~  g(p) h(q) ~\frac{\sqrt{\pZ \qZ}}{|p-q|}
\\
+
\int_{\mathbb{R}^3}dq  \int_{B(0,R),\rho \geq \frac{1}{8}}dp~  g(p) h(q) \left( \frac{\pZ}{\qZ} + \frac{\qZ}{\pZ} \right).
\end{multline*}
We will estimate these upper bound integrals one at a time.  Notice that
\begin{multline*}
\int_{\mathbb{R}^3}dq  \int_{B(0,R)}dp~  g(p) h(q) \left( \frac{\pZ}{\qZ} + \frac{\qZ}{\pZ} \right)
\\
\le
\| g \|_{L^1_1(B(0,R))} \| h \|_{L^1_{-1}(\threed)} 
+
\| g \|_{L^1_{-1}(B(0,R))} \| h \|_{L^1_1(\threed)}.
\end{multline*}
Further on $\rho < \frac{1}{8}$ as in \eqref{compare.pq.est} then $\pZ \leq 2$ implies that $\qZ\leq 5$ and we further have $|p-q| \leq 2$.  We conclude using Young's inequality that
\begin{multline*}
\int_{\mathbb{R}^3}dq  \int_{B(0,R), ~\rho < \frac{1}{8}, ~\pZ \leq 2 }dp~  g(p) h(q) ~\frac{\sqrt{\pZ \qZ}}{|p-q|}
\\
\lesssim
\int_{\mathbb{R}^3}dq  \int_{B(0,R), |p-q| \leq 2 }dp~  g(p) h(q) ~\frac{1}{|p-q|}
\\
\lesssim
\| h \|_{L^1(\threed)}
\| g \|_{L^{r'}(B(0,R))}
\left\| \frac{1}{|\cdot|} 1_{\{|\cdot|\leq 2\}} \right\|_{L^r(\threed)}.
\end{multline*}
Here $1=\frac{1}{r}+\frac{1}{r'}$ and we require $1\le r < 3$ to conclude that $\left\| \frac{1}{|\cdot|} 1_{\{|\cdot|\leq 2\}} \right\|_{L^r(\threed)} \lesssim 1$.  Similarly we have
\begin{multline*}
\int_{\mathbb{R}^3}dq  \int_{B(0,R), |p-q| \leq 2 }dp~  g(p) h(q) ~\frac{1}{|p-q|}
\\
\lesssim
\| h \|_{L^{r'}(\threed)}
\| g \|_{L^{1}(B(0,R))}
\left\| \frac{1}{|\cdot|} 1_{\{|\cdot|\leq 2\}} \right\|_{L^r(\threed)}.
\end{multline*}
Thus for this term we conclude
\begin{multline*}
\int_{\mathbb{R}^3}dq  \int_{B(0,R), ~\rho < \frac{1}{8}, ~\pZ \leq 2 }dp~  g(p) h(q) ~\frac{\sqrt{\pZ \qZ}}{|p-q|}
\\
\lesssim
\min\{
\| h \|_{L^1(\threed)}
\| g \|_{L^{r'}(B(0,R))},
~
\| h \|_{L^{r'}(\threed)}
\| g \|_{L^1(B(0,R))}
\}.
\end{multline*}
This holds as long as we have $r' \in (3/2, \infty]$.

Lastly on the region $\rho < \frac{1}{8}$ with $\pZ \ge 2$, then we further have $\qZ \ge 1$ and also $\pZ \approx \qZ$.  See \eqref{D2-comp}.  
If $|p-q| \ge 1$ also then we have
\begin{multline*}
\int_{\mathbb{R}^3}dq  \int_{B(0,R), \rho < \frac{1}{8}, \pZ \ge 2, |p-q| \ge 1}dp~  g(p) h(q) ~\frac{\sqrt{\pZ \qZ}}{|p-q|}
\\
\lesssim
\min\{
\| h \|_{L^1(\threed)}
\| g \|_{L^1_1(B(0,R))},
\| h \|_{L^1_1(\threed)}
\| g \|_{L^1(B(0,R))}
\}.
\end{multline*}
On the other hand when further $|p-q| \le 1$ then we have
\begin{multline*}
\int_{\mathbb{R}^3}dq  \int_{B(0,R), \rho < \frac{1}{8},  \pZ \ge 2, |p-q| \le 1}dp~  g(p) h(q) ~\frac{\sqrt{\pZ \qZ}}{|p-q|}
\\
\lesssim
\| h \|_{L^{r'}(\threed)}
\| g \|_{L^{1}_1(B(0,R))}
\left\| \frac{1}{|\cdot|} 1_{\{|\cdot|\leq 1\}} \right\|_{L^r(\threed)}.
\end{multline*}
Similarly we also have
\begin{multline*}
\int_{\mathbb{R}^3}dq  \int_{B(0,R), \rho < \frac{1}{8},  \pZ \ge 2, |p-q| \le 1}dp~  g(p) h(q) ~\frac{\sqrt{\pZ \qZ}}{|p-q|}
\\
\lesssim
\| h \|_{L^{1}_1(\threed)}
\| g \|_{L^{r'}(B(0,R))}
\left\| \frac{1}{|\cdot|} 1_{\{|\cdot|\leq 1\}} \right\|_{L^r(\threed)}.
\end{multline*}
Therefore again when $r' \in (3/2, \infty]$ then we further have
\begin{multline*}
\int_{\mathbb{R}^3}dq  \int_{B(0,R), \rho < \frac{1}{8},  \pZ \ge 2, |p-q| \le 1}dp~  g(p) h(q) ~\frac{\sqrt{\pZ \qZ}}{|p-q|}
\\
\lesssim
\min\{
\| h \|_{L^{1}_1(\threed)}
\| g \|_{L^{r'}(B(0,R))},
\| h \|_{L^{r'}(\threed)}
\| g \|_{L^{1}_1(B(0,R))}
\}.
\end{multline*}
Collecting all of these estimates completes the proof.
\end{proof}

Now we integrate \eqref{weak.dt.fepsilon} in time to obtain that 
\begin{multline}\notag
\left| \int_{\mathbb{R}^3} dp ~ f^{\epsilon}(t) \varphi  - \int_{\mathbb{R}^3} dp ~ f^{\epsilon}(s) \varphi  \right|
\\
\lesssim
\|\varphi \|_{W^{2,\infty}}
\int_s^t d\tau
\int_{\mathbb{R}^3}
 \int_{B(0,R)} 
 f^{\epsilon}(q) f^{\epsilon}(p) \left|  \Phi^{ij}(p,q) \right| dqdp
\\
+
\|\varphi \|_{W^{2,\infty}}
\int_s^t d\tau
\int_{\mathbb{R}^3}
 \int_{B(0,R)}  
 f^{\epsilon}(p) f^{\epsilon}(q)
\Lambda(p,q) (\rho +2) |p-q|^2 dqdp
\\
+
\epsilon 
(t-s) 
\|\varphi \|_{W^{2,\infty}}\| f^{\epsilon}_0\|_{L^1}.
\end{multline}
Now we use Lemma \ref{upper.bd.lemma} with $r'=\frac{3}{2}+\epsilon$ for a small $\epsilon>0$ to estimate the two terms in the middle above.  Then we interpolate $\| f^{\epsilon}\|_{L^{r'}(\threed)} \le \| f^{\epsilon}\|_{L^{1}(\threed)}^\theta \| f^{\epsilon}\|_{L^{3}(\threed)}^{1-\theta} $.  After using a H{\"o}lder inequality in the time integral we obtain 
\begin{equation}\notag
\left| \int_{\mathbb{R}^3} dp ~ f^{\epsilon}(t) \varphi  - \int_{\mathbb{R}^3} dp ~ f^{\epsilon}(s) \varphi  \right|
\le C
(t-s)^\theta 
\|\varphi \|_{W^{2,\infty}} 
+
C (t-s) 
\|\varphi \|_{W^{2,\infty}},
\end{equation}
where the constant $C>0$ depends on $\| f^{\epsilon}\|_{L^{1}(\threed)}$,  $\| f^{\epsilon}\|_{L^{1}_1(\threed)}$, and $\| f^{\epsilon}\|_{L^1([0,T];L^{3}(\threed))}$ which quantities have been shown to be uniformly bounded.

Therefore the sequence $\left( \int f^\epsilon \varphi dp\right)_{\epsilon>0}$ is uniformly bounded and equicontinuous in $\mathcal{C}([0,T])$.  By the Arzela-Ascoli theorem this sequence is further relatively compact in $\mathcal{C}([0,T])$.  We can conclude from the convergence of $(f^\epsilon)_{\epsilon >0}$  to $f$  that $\left( \int f \varphi dp\right)$ is the unique cluster point of  $\left( \int f^\epsilon \varphi dp\right)_{\epsilon>0}$.

\subsection{Weak solutions}
Given $f_0 \in L^1_s(\threed) \cap L \log L (\threed)$ for some $s>1$ we choose smooth initial data $f_0^\epsilon(p)$ as described in Section \ref{sec.regularized.sec} that satisfies  $f_0^\epsilon \in L^1_s(\threed)$ and $f_0^\epsilon \to f_0$ as $\epsilon \to 0$ in $L^1_s(\threed) \cap L \log L (\threed)$.

Then the calculations in the proof of Theorem \ref{prop-mom} allow us to conclude that 
\begin{equation}
\int_{\threed} f^\epsilon(t,p) (1+|p|^2)^{s/2} dp \le C,
\label{uniform.higher.bound}
\end{equation}
where the constant $C$ is uniform in $[0,T]$ and in $\epsilon>0$.  Then by the pointwise convergence of $(f^\epsilon)_{\epsilon >0}$ to $f$  and Fatou's lemma we conclude that
$$
\int_{\threed} f(t,p) (1+|p|^2)^{s/2} dp \le C.  
$$
Now we look at the weak formulation.  

First to simplify the notation we define 
\begin{equation}\notag
\tilde{b}^i_\epsilon (p,q)
\eqdef
\partial_{p_j}\Phi^{ij}_{\epsilon}(p,q) - \partial_{q_j}\Phi^{ij}_{\epsilon}(p,q),
\end{equation}
and
\begin{equation}\notag
\tilde{b}^i  (p,q)
\eqdef
\partial_{p_j}\Phi^{ij}(p,q) - \partial_{q_j}\Phi^{ij}(p,q).
\end{equation}

From the \eqref{weak.form.landau.kernel}, for the approximate problem \eqref{truncatedCauchy.n}, as in \eqref{weak.formA.landau} and \eqref{weak.formC.landau}, we have the following weak formulations
\begin{multline}\notag
- \int_{\mathbb{R}^3} dp ~ f_0^{\epsilon} \varphi(0,p) - 
\int_0^T dt ~\int_{\mathbb{R}^3} dp ~ f^{\epsilon} \partial_t\varphi
\\
=
\frac{1}{2}\int_0^T dt ~
\int_{\mathbb{R}^3} \int_{\mathbb{R}^3} f^{\epsilon}(q) f^{\epsilon}(p) \Phi^{ij}_{\epsilon}(p,q)
\left( \partial_{p_j}\partial_{p_i}\varphi(p) +  \partial_{q_j}\partial_{q_i}\varphi(q)\right) dqdp
\\
+
\int_0^T dt ~\int_{\mathbb{R}^3}
 \int_{\mathbb{R}^3} f^{\epsilon}(p) f^{\epsilon}(q)
\tilde{b}^i_\epsilon (p,q)
\left( \partial_{p_i}\varphi(p)  - \partial_{q_i}\varphi(q)\right) dqdp
\\
+
\epsilon 
\int_0^T dt ~ \int_{\mathbb{R}^3} dp ~  f^{\epsilon}(p) ~ \partial_{p_i}\partial_{p_i}\varphi(p).
\end{multline}
And with no approximation the weak formulation is:
\begin{multline}\notag
- \int_{\mathbb{R}^3} dp ~ f_0 \varphi(0,p) - 
\int_0^T dt ~\int_{\mathbb{R}^3} dp ~ f \partial_t\varphi
\\
=
\frac{1}{2}\int_0^T dt \int_{\mathbb{R}^3}
 \int_{\mathbb{R}^3} f(q) f(p) \Phi^{ij}(p,q)
\left( \partial_{p_j}\partial_{p_i}\varphi(p) +  \partial_{q_j}\partial_{q_i}\varphi(q)\right) dqdp
\\
+\int_0^T dt ~\int_{\mathbb{R}^3}
 \int_{\mathbb{R}^3} f(p) f(q)
\tilde{b}^i (p,q)  
\left( \partial_{p_i}\varphi(p)  - \partial_{q_i}\varphi(q)\right) dqdp.
\end{multline}
Note that the first two terms on the left side, and the last term containing the Laplacian, clearly converge as $\epsilon \to 0$.  Now we consider the remaining two terms.  First we have
\begin{multline}\notag
\left|  
\int_0^T dt\int_{\mathbb{R}^3}
 \int_{\mathbb{R}^3} f^{\epsilon}(q) f^{\epsilon}(p) \Phi^{ij}_{\epsilon}(p,q)
\left( \partial_{p_j}\partial_{p_i}\varphi(p) +  \partial_{q_j}\partial_{q_i}\varphi(q)\right) dqdp
\right.
\\
-
\left.
\int_0^T dt \int_{\mathbb{R}^3}
 \int_{\mathbb{R}^3} f(q) f(p) \Phi^{ij}(p,q)
\left( \partial_{p_j}\partial_{p_i}\varphi(p) +  \partial_{q_j}\partial_{q_i}\varphi(q)\right) dqdp
 \right|
\\
\lesssim
\left|
\int_0^T 
\int_{B(0,R')}
 \int_{B(0,R)}
 \left( f(q) f(p) \Phi^{ij}(p,q)
 -
 f^{\epsilon}(q) f^{\epsilon}(p) \Phi^{ij}_{\epsilon}(p,q) \right)
\partial_{p_j}\partial_{p_i}\varphi(p)  
 \right|
 \\
 +
\| \varphi\|_{W^{2,\infty}}
\int_0^T dt\int_{B(0,R')^c} dq
 \int_{B(0,R)} dp ~
 f(q) f(p)  \left| \Phi^{ij}(p,q) \right|
  \\
 +
\| \varphi\|_{W^{2,\infty}}
\int_0^T dt\int_{B(0,R')^c} dq
 \int_{B(0,R)} dp ~
 f^{\epsilon}(q) f^{\epsilon}(p) \left| \Phi^{ij}_{\epsilon}(p,q) \right|.
\end{multline}
The last two terms in the upper bound, after using Lemma \ref{upper.bd.lemma}, will go to zero as $R' \to\infty$ because of the uniform higher moment bound in \eqref{uniform.higher.bound}.
Now the first term in the upper bound converges to zero by the strong convergence of $\Phi^{ij}_{\epsilon}(p,q)$ to $\Phi^{ij}(p,q)$ as in \eqref{r.convergence.est} and again using Lemma \ref{upper.bd.lemma} and the strong convergence of $ f^{\epsilon}$ to $f$ that we have previously established.

The convergence in the last term above involving $\tilde{b}^i (p,q)$ can be shown similarly.
This completes the proof of Theorem \ref{weak.sol.thm}.


%

\end{document}